\documentclass[11pt]{amsart}

\usepackage{pb-diagram}

\usepackage{amsfonts}
\usepackage{amsmath}
\usepackage{amssymb}

\newcommand{\tr}{\textnormal{tr}}
\newcommand{\ric}{\textnormal{Ric}}

\newcommand{\BB}{\mathcal{B}}

\newcommand{\dbar}{\overline{\partial}}

\newcommand{\ddt}[1]{\frac{\partial #1}{\partial t}}

\newcommand{\zz}{\mathbf{z}}

\newcommand{\OO}{\mathcal{O}}
\newcommand{\EE}{\mathcal{E}}

\newcommand{\AAA}{\mathcal{A}}

\newcommand{\ddbar}{\sqrt{-1}\partial\dbar}

\newtheorem{theorem}{Theorem}[section]
\newtheorem{proposition}{Proposition}[section]
\newtheorem{lemma}{Lemma}[section]
\newtheorem{conjecture}{Conjecture}[section]
\newtheorem{definition}{Definition}[section]
\newtheorem{corollary}{Corollary}[section]

\newcommand{\PP}{\mathbb{P}}
\newcommand{\CC}{\mathbb{C}}

\begin{document}

\title{Ricci flow and birational surgery}

\author{Jian Song}

\address{Department of Mathematics, Rutgers University, Piscataway, NJ 08854}

\email{jiansong@math.rutgers.edu}

\thanks{Research supported in
part by National Science Foundation grants DMS-0847524 and  a Sloan Foundation Fellowship.}

\begin{abstract} We study the formation of finite time singularities of the K\"ahler-Ricci flow in relation to high codimensional birational surgery in algebraic geometry. We show that the K\"ahler-Ricci flow on an $n$-dimensionl K\"ahler manifold  contracts  a complex submanifold $\mathbb{P}^m$ with normal bundle $\oplus_{j=1}^{n-m}\mathcal{O}_{\mathbb{P}^m}(-a_j)$ for $a_j\in \mathbb{Z}^+$ and $\sum_{j=1}^{n-m} a_j \leq m$ in Gromov-Hausdorff topology with suitable initial K\"ahler class. We also show that the K\"ahler-Ricci flow resolves a family of isolated singularities uniquely in Gromov-Hausdorff topology. In particular, we construct global and local examples of metric flips by the K\"ahler-Ricci flow as a continuous path in Gromov-Hausdorff topology.

\end{abstract}

\maketitle

{\footnotesize  \tableofcontents}


\section{Introduction}

In his seminal work \cite{H1, H2}, Hamilton introduced the Ricci flow to study the global structures and classification of Riemannian manifolds. In general, the Ricci flow may develop singularities in finite time. Hamilton conjectured that in the three dimensional case, the Ricci flow should break the manifold into pieces near singular times by topological surgery and the flow can be continued on the new manifolds. Such a program for Ricci flow with surgery eventually leads to the complete proof of Thurston's geometrization conjecture for $3$-manifolds by Perelman's fundamental work \cite{P1, P2, P3, Cao2, KL, MT}. It was further suggested by Perelman that the Ricci flow should carry out surgeries through singularities in some intrinsic and unique way as referred to as canonical surgery by the Ricci flow.

There has been extensive study  for the Ricci flow in the setting of K\"ahler manifolds, following Yau's  solution to the Calabi conjecture \cite{Y1}. The Ricci flow preserves the K\"ahler condition, i.e., if the initial metric is K\"ahler, then the smooth solution of the Ricci flow is also K\"ahler. Let $(X, \omega_0)$ the a compact K\"ahler manifold with $\dim_{\CC} X =n \geq 2$ and $g_0$ be a K\"ahler metric on $X$. We consider the K\"ahler-Ricci flow,
\begin{equation}\label{krflow}
\ddt{}g = - \ric(g), ~g|_{t=0} = g_0,
\end{equation}
where $g=g(t)$ is the metric associated to its K\"ahler form $\omega(t)$ and $\ric(\omega) = -\ddbar \log \omega^n$ is the Ricci curvature of $\omega$.  The flow always has short time existence and admits a smooth solution on $[0, t]$, if and only if, the
cohomology class of $\omega(t)$ given by
$[\omega(t)] = [\omega_0] + t  [K_X]$ is K\"ahler. The first singular time  $T$ is characterized in \cite{TZha}  by
\begin{equation} \label{T}
T = \sup \{ t \in \mathbb{R} \ | \ [\omega_0] + t  [K_X] >0 \}.
\end{equation}
Clearly $T$ depends  only on $X$ and the initial K\"ahler class $[\omega_0]$, satisfying  $0 < T \le \infty$.

The necessary condition for a K\"ahler manifold $X$ to admit a smooth K\"ahler-Einstein metric is that the first Chern class of $X$ is definite or vanishing.  It was shown in \cite{Cao1} that the K\"ahler-Ricci flow always converges exponentially fast to a K\"ahler-Einstein metric if the first Chern class is negative or vanishing. When the first Chern class is positive, the convergence to a compact K\"ahler-Ricci soliton  is known under certain assumptions \cite{P4, SeT, TZhu1, PS, PSSW1, PSSW2, CW, TZhu2, TZhu3, CS}. One would hope that the K\"ahler-Ricci flow should deform any initial K\"ahler metric to a K\"ahler-Einstein metric, however, most K\"ahler manifolds do not admit definite or vanishing first Chern class and so the flow will in general develop singularities. An  analytic minimal model program with Ricci flow was laid out in \cite{SoT3} to study such  formation of singularities for the K\"ahler-Ricci flow on algebraic varieties. It can be viewed as an analogue of Thurston's geometrization conjecture for projective varieties. It is conjectured in \cite{SoT3, T2}  that the K\"ahler-Ricci flow will either deform a projective variety $X$ of nonnegative Kodaira dimension, to its minimal model via finitely many divisorial metric contractions and metric flips in Gromov-Hausdorff topology, then eventually converge to a unique canonical metric of Einstein type on its unique canonical model. The existence and uniqueness is proved in \cite{SoT3} for the analytic solutions of the K\"ahler-Ricci flow on algebraic varieties with log terminal singularities. Furthermore, the K\"ahler-Ricci flow can be analytically and uniquely extended through divisorial contractions and flips \cite{SoT3}. The solution is in fact smooth outside the singularities of the underlying varieties and where the algebraic surgery takes place. However, very little is understood about the surgery in global and local Riemannian geometry. We believe that the algebraic surgery in birational geometry, analytic surgery by the parabolic complex Monge-Amp\`ere equation and the Riemannian geometric surgery are all equivalent via Ricci flow:
$$ Algebraic~ surgery~ \Longleftrightarrow ~Analytic ~ surgery~\Longleftrightarrow ~Riemannian~ surgery . $$
The notion of a canonical surgical contraction was first introduced in \cite{SW2}, in which it was  shown that the K\"ahler-Ricci flow indeed performs a canonical surgery by contracting smooth divisors $\PP^{n-1}$ with normal bundle $\OO_{\PP^{n-1}}(-1)$ and the flow can be uniquely extended on the blow-down manifold in Gromov-Hausdorff topology. As an application, it implies that the K\"ahler-Ricci flow on a K\"ahler surface with any initial K\"ahler metric will perform a sequence of canonical surgeries contracting holomorphic embedded $S^2$ with self-intersection number $-1$, until in finite time, either the minimal model is obtained or the volume of the manifold tends to zero \cite{SW2, SW4}.  In \cite{SW3}, it is further shown that the K\"ahler-Ricci flow blows down disjoint exceptional divisors $\PP^{n-1}$ with normal bundle $\OO_{\PP^{n-1}}(-k)$ to orbifold points whenever $0<k<n$. The hermitian extension of such surgical contractions is studied in \cite{StT, TW1, TW2} for the  Chern-Ricci flow.  In \cite{LT}, an alternative approach to understand the K\"ahler-Ricci flow through singularities is proposed in the frame work of K\"ahler quotients by transforming the parabolic complex Monge-Amp\`ere equation into an elliptic  $V$-soliton equation.

The high codimensional contraction by the K\"ahler-Ricci flow is much more complicated because whenever the exceptional locus of the contraction has codimension greater than 1, the blow-down variety must be singular and in fact, the singularity cannot be $\mathbb{Q}$-Gorenstein. The first example was constructed in  \cite{SY} for the K\"ahler-Ricci flow with Calabi symmetry on projectivized bundles over a projective space. In such an example, the K\"ahler-Ricci flow contracts an holomorphic embedded $\PP^m$ with normal bundle $\OO_{\PP^m}(-1)^{\oplus (n-m)}$ to a singular  point, if $n/2\leq m \leq n-2$. The blow-down variety is a projective cone over $\PP^m \times \PP^{n-m-1}$ in $\PP^{(m+1)(n-m)-1}$ via the Segre embedding.  Our first main result is to  generalize the above example and apply the K\"ahler-Ricci flow to contract embedded $\PP^m$ with negative normal bundle to a point, in Gromov-Hausdorff topology.

 \begin{theorem} \label{main1} Let $X$ be a projective manifold of $\dim_{\CC} X =n$ and let $g(t)$ be a smooth solution of the K\"ahler-Ricci flow (\ref{krflow}) for $t\in [0, T)$, starting from a smooth K\"ahler metric $g_0$ with $[g_0] \in H^{1,1}(X, \mathbb{R})\cap H^2(X, \mathbb{Q})$.  Suppose that

 \begin{enumerate}

 \item $E$ is a disjoint union of complex submanifolds $\mathbb{P}^{n_i}$ of $X$  with normal bundle
 $$\oplus_{j=1}^{n-n_i} \mathcal{O}_{\mathbb{P}^{n_i}}(-a_{i,j}), ~a_{i,j}\in \mathbb{Z}^+ ,  ~\sum_{j=1}^{n-n_i} a_{i,j} \leq n_i  $$
for $i=1, ..., k$;

\medskip

\item   the limiting K\"ahler class $\lim_{t\rightarrow T} [g(t)] $ is the pullback of an ample class on the projective variety $Y$ from the birational morphism $\pi: X \rightarrow Y$ by contracting $E$.

\end{enumerate}
Then the following holds.
\begin{enumerate}

\item $g(t)$ converges to a smooth K\"ahler metric $g(T)$ on $X\setminus E$ in $C^\infty(X\setminus E)$ as $t\rightarrow T$;

\medskip

\item The metric completion of $(X\setminus E, g(T))$ is a compact metric length space homeomorphic to $Y$. We denote it by $(Y, d_T)$;

\medskip

\item $(X, g(t))$ converges in Gromov-Hausdorff topology to $(Y, d_T)$ as $t\rightarrow T$.

\end{enumerate}

 \end{theorem}

We remark that the blown-down variety $Y$ in Theorem \ref{main1} must be singular if $E$, the exceptional locus of the contraction,  has a component of codimension greater than $1$. The condition (2) in the assumption of Theorem \ref{main1} can be easily met, for example, one can simply choose the initial K\"ahler class to be $\OO_Y(1) - \epsilon K_X$ for sufficiently small $\epsilon>0$. In particular, the conclusion in Theorem \ref{main1} only depends on the initial K\"ahler class $[g_0]$ and does not depend on the choice of the smooth K\"ahler metric in $[g_0]$.

The Ricci flow is a nonlinear parabolic equation on a Riemannian manifold. One property for the heat equation is to smooth out the initial data. We should view the solution of the Ricci flow as a pair $(X_t, g(t))$, because the evolving Riemannian metrics would possibly deform the underlying variety by creating and resolving singularities. When the divisorial contractions for a smooth divisor $\PP^{n-1}$ with normal bundle $\OO_{\PP^{n-1}}(-1)$ take place, it is shown in \cite{SW2}, the flow converges to a singular K\"ahler metric on the smooth blow-down manifold with isolated metric singularities, however, the flow can be uniquely extended beyond the singular time by smoothing out the singular metric instantly.   In such a case, the K\"ahler-Ricci flow resolves the metric singularity without any surgery on the underlying manifold. When a high codimensional contraction takes place, the singularities of the blow-down variety are not $\mathbb{Q}$-Gorenstein and so the canonical divisor is not a Cartier $\mathbb{Q}$-divisor. In this case, one cannot properly define  Ricci curvature and therefore the K\"ahler-Ricci flow can not be extended  through singularity without replacing the underlying variety by a unique model with minimal resolution of singularities.
Various resolutions of singularities are constructed in \cite{SY} using Ricci curvature or Ricci flow for a family of projective varieties with large symmetry. Our next result is to show that the K\"ahler-Ricci flow indeed resolves singularities in Gromov-Hausdorff topology uniquely.

\begin{theorem} \label{main2}  Let $Y$ be a projective variety of $\dim_{\CC} Y=n$ with isolated singularities $p_1, ..., p_k$. Suppose that

\begin{enumerate}

\item $\pi: X \rightarrow Y$ is a resolution of singularities along $p_1, ..., p_k$ with $$\pi^{-1} (p_i) = \mathbb{P}^{n_i}, ~1\leq n_i \leq n-2 ;$$

\medskip

\item the normal bundle of $\mathbb{P}^{n_i}$ in $X$ is $$\oplus_{j=1}^{n-n_i} \mathcal{O}_{\mathbb{P}^{n_i}}(-a_{i,j}), ~a_{i,j}\in \mathbb{Z}^+ ,  ~\sum_{j=1}^{n-n_i} a_{i,j} > n_i . $$

\end{enumerate}
Let $g_0$ be a smooth K\"ahler metric on $Y$, i.e., $g_0$ is locally the restriction of a smooth K\"ahler metric for a local embedding of $Y$ in some $\CC^N$. Then there exists a unique smooth solution $g(t)$ of the K\"ahler-Ricci flow  on $X$ for $t\in (0, T)$ for some $T\in (0,  \infty]$ satisfying

\begin{enumerate}

\item $g(t)$ converges to  $g_0$ on $X\setminus E$ in $C^\infty(X\setminus E)$, where $E=\cup_{i=1}^k \mathbb{P}^{n_i}$;

\medskip
\item $(X, g(t))$ converges in Gromov-Hausdorff topology to $(Y, g_0)$ as $t\rightarrow 0$.

\end{enumerate}

\end{theorem}

Both Theorem \ref{main1} and Theorem \ref{main2} can be generalized to a more general setting of Morelli-Wlodarczyk birational cobordisms for Mumford's quotients \cite{Mo, W}. We consider the $\CC^*$ action on $\CC^{m+l+2}$ defined by
$$ \CC^*: (\lambda, (x_0, ..., x_m; y_0, ..., y_l)) \rightarrow ( \lambda^{-a_0} x_0, \lambda^{-a_1} x_1, ..., \lambda^{-a_m} x_m; \lambda^{b_0} y_0, ..., \lambda^{b_l} y_l ), $$
where $(a_0, ..., a_m, b_0, ..., b_l) \in \mathbb{Z}^{+}$ and $g.c.d(a_0, ..., a_m, b_0, ...,  b_l) =1. $
We define
$$\EE^-=\left( \CC^{m+l+2}\setminus \{ x=0\} \right) / \CC^*, ~ \EE^+ =\left( \CC^{m+l+2}\setminus \{y=0\} \right) / \CC^*. $$
When $a_0=a_1=...=a_m=b_0=...=b_l=1$,
if $l=0$,
$$\EE^- = \OO_{\PP^m}(-1), ~\EE^+ =\CC^{m+1}, $$
and if $l\geq 1$,
$$\EE^- = \OO_{\PP^m}(-1)^{\oplus (l+1)}, ~\EE^+ = \OO_{\PP^l}(-1)^{\oplus (m+1)}. $$
In section 5, we prove much more general results (Theorem \ref{gen1} and Theorem \ref{gen2}), showing that the conclusions in Theorem \ref{main1} and Theorem \ref{main2} both hold whenever locally, the neighborhood of the algebraic exceptional locus is isomorphic to a flip induced by  Mumford's quotients. We also remark that our techniques can be easily modified to obtain the metric flops and metric geometric transitions for Calabi-Yau orbifolds in the local setting of Mumford's quotients as in \cite{RZ1, So2, RZ2}.

More interestingly, we could like to strengthen Theorem \ref{main2} for more singular initial metrics so that one can combine Theorem \ref{main1} and Theorem \ref{main2} to produce global metric flips by the K\"ahler-Ricci flow.  Let us recall the definition for a flip (c.f. \cite{KMM, KM}).

\begin{definition} Let $\phi_-:X \rightarrow Y $ be a birational morphism with exceptional locus of codimension greater than $1$ such that $-K_X$ is $\mathbb{Q}$-Cartier and $\phi_-$-ample. Then a variety $X^+$ together with a birational morphism $\phi_+: X^+\rightarrow Y$ is called a flip of $X^-$ if $K_{X^+}$ is $\mathbb{Q}$-Cartier and $\phi_+$-ample as in the following diagram
$$
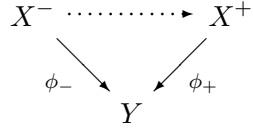
\begin{diagram}
\node{X^-} \arrow{se,b,}{\phi_-}  \arrow[2]{e,t,..}{ }     \node[2]{X^+} \arrow{sw,r}{\phi_+} \\
\node[2]{Y}
\end{diagram}$$

\end{definition}

The exceptional locus of $\phi_-$ and $\phi_+$ are where $\phi_-$ and $\phi_+$ are not isomorphisms. Both $\phi_-$ and $\phi_+$ contract subvarieties of codimension greater than $1$, hence $Y$ always admits singularities, in particular, $K_Y$ is not a $\mathbb{Q}$-Cartier divisor. We hope to relate flipping surgery in birational geometry to metric surgery by the Ricci flow in Riemannian geometry. First, we define the following surgical metric flip by the K\"ahler-Ricci flow similar to the surgical metric divisorial contraction defined    in \cite{SW2}.

\begin{definition} \label{defn}

Suppose $X^-$ and its flip $X^+$ are smooth projective varieties. The exceptional locus of $\phi_-$ and $\phi_+$, $E^-$  and $E^+$, are subvarieties of $X^-$ and $X^+$ of codimension $\geq 2$.  The K\"ahler-Ricci flow  is said to perform a  {\bf  surgical metric flip} at $t=T$ as  in the following the diagram
$$
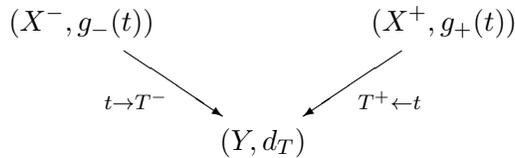
\begin{diagram}
\node{(X^-, g_-(t)) } \arrow{se,b,}{t\rightarrow T^-}      \node[2]{ (X^+, g_+(t))} \arrow{sw,r}{T^+\leftarrow t} \\
\node[2]{(Y, d_T)}
\end{diagram}$$
if the following holds.
\begin{enumerate}

\item  There exists a smooth solution $g(t)$ on $X^-$ for $t\in (T_{-}, T)$  such that

\begin{itemize}

\medskip

\item the metrics $g(t)$ converge to a smooth K\"ahler metric $g(T)$ on $X^- \setminus E^-$ in $C^\infty(X^-\setminus E^-)$ as $t\rightarrow T^-$;

\medskip

\item the metric completion of $(X^-\setminus E^-, g(T))$ is a compact metric length space homeomorphic to the projective variety $Y$. We denote it by $(Y, d_T)$;

\medskip

\item $(X^-, g(t))$ converges to $(Y, d_T)$ in Gromov-Hausdorff topology as $t\rightarrow T^-$.

\end{itemize}

\medskip

\item There exists a smooth solution $g(t)$ on $X^+$ for $t\in (T, T_+)$ such that

\begin{itemize}

\medskip

\item the metrics $g(t)$ converge to $g(T)$ on $X^+ \setminus E^+$ in $C^\infty(X^-\setminus E^+)$ as $t\rightarrow T^+$;

\medskip

\item $(X^+, g(t))$ converges to $(Y, d_T)$ in Gromov-Hausdorff topology as $t\rightarrow T^+$.

\end{itemize}

\end{enumerate}
\end{definition}

We would like to remark that Definition \ref{defn} can be modified when $X^-$ and $X^+$ have orbifold and more generally log terminal singularities.

Unfortunately, we are unable to show at the moment that the K\"ahler-Ricci flow perform a surgical metric flip in the assumption of Theorem \ref{main1} or more general assumptions  in Theorem \ref{gen1} because Theorem \ref{main2} requires the initial metric to be smooth.  However, we can apply the estimates in the proof of Theorem \ref{main1} and Theorem \ref{main2} to a family of projective manifolds with large symmetry considered in \cite{SY} and show that the K\"ahler-Ricci flow indeed performs surgical metric flips if the initial K\"ahler metric is appropriately chosen.  To be more precise, we consider a projective toric manifold
$$X_{m,l}= \mathbb{P}(\mathcal{O}_{\mathbb{P}^m} \oplus \mathcal{O}_{\mathbb{P}^m}(-1)^{\oplus (l+1)}) $$ for $m\geq 0$ and $l\geq 1$ and let $n=\dim_{\CC} X_{m, l} =m+l+1$. $X_{m,0}$ is exactly $\mathbb{P}^{m+1}$ blown up at one point. $X_{m,l}$ does not admit a definite or vanishing first Chern class when $m\leq l$ and $X_{m,l}$ is Fano if and only if $m>l$. $X_{m,l}$ has a special subvariety $E_{m, l}$ of codimension $m+1$, defined as the zero section of the projection $\mathbb{P}(\mathcal{O}_{\mathbb{P}^m} \oplus \mathcal{O}_{\mathbb{P}^m}(-1)^{\oplus (l+1)})\rightarrow \mathbb{P}^m$.  There exists a morphism
\begin{equation}
\Phi_{m,l}: X_{m,l} \rightarrow \mathbb{P}^{(m+1)(l+1)-1}
\end{equation} which is an immersion on $X_{m, l} \setminus E_{m, l}$ and contracts $E_{m, l}$ to a point. $Y_{m,l}$, the image of $X_{m,l}$ via $\phi_{m,l}$ is smooth if and only if $l=0$ and then $Y_{m,0}$ is simply $\mathbb{P}^{m+1}$. When $l\geq 1$, $Y_{m,l}$ has an isolated cone singularity  where $E_{m, l}$ is contracted. In particular, $Y_{m, l}=Y_{l,m}$ is the projective cone in $\mathbb{P}^{(m+1)(l+1)-1}$ over $\mathbb{P}^{m}\times \mathbb{P}^l $ via the Segre embedding for $l\geq 1$. It is also well-known that $X_{m,l}$ and $X_{l,m}$ are birationally equivalent for $l\geq 1$, and differ by a flip for $m\neq l$ and a flop when $m=l$. We can define a family of $U(l+1)$-invariant K\"ahler metrics on $X_{m, l}$ satisfying the so-called Calabi symmetry \cite{C1} (see section 2.1).

\begin{theorem} \label{main3} Let $g(t)$ be the  solution of the K\"ahler-Ricci flow (\ref{krflow}) on $X_{m,l}$ with $0<l <m$ and an initial K\"ahler metric $g_0$ satisfying the Calabi symmetry. The flow must develop singularity at $t=T$ for some $T<\infty$. Let $V(t)=\int_{X_{m, l}} dVol_{g(t)}$ be the volume of $(X_{m,l}, g(t))$.

\medskip

\begin{enumerate}

\item If $\limsup_{t\rightarrow T}\frac{ V(t) }{ (T-t)^{m-l}} = \infty$, then the Ricci flow perform a surgical metric flip at $t=T$ as in Definition \ref{defn}.

\medskip

\item If $\limsup_{t\rightarrow T}\frac{ V(t) }{ (T-t)^{m-l}}  \in (0, \infty)$, then the Ricci flow converges to $\PP^m$ coupled with a multiple of the Fubini-Study metric on $\PP^m$ in Gromov-Hausdorff topology as $t\rightarrow T$.

\medskip

\item If $\limsup_{t\rightarrow T}\frac{ V(t) }{ (T-t)^{m-l}}  =0$, then the Ricci flow becomes extinct at $t= T$.

\end{enumerate}

\end{theorem}

We remark that (2) and (3) in Theorem \ref{main3} were proved in \cite{SY}.  The K\"ahler-Ricci flow on $X_{m,l}$ will eventually converge to a point in finite time for any initial K\"ahler metric satisfying the Calabi symmetry, after a flip or collapsing to $\PP^m$. The Gromov-Hausdorff surgery by the K\"ahler-Ricci flow reflects the deformation of the global structure in both algebraic and differential geometry. However, the Gromov-Hausdorff topology is rather weak and we would like to understand the asymptotic microlocal structure of the surgery. This is usually studied by blowing up the solution near the singularity by parabolic dilation.    In the case of $\PP^n$ blow-up at one point, if the initial K\"ahler metric is invariant under $U(n)$ action and the total volume does not tend to $0$, it is shown in \cite{So2} that the K\"ahler-Ricci flow must develop Type-I singularity and the blow-up limit is a complete shrinking gradient Ricci soliton on a K\"ahler manifold diffeomorphic to $\CC^n$ blow-up at one point.  This suggests that the surgery performed by the Ricci flow is indeed the most optimal procedure to deform the underlying variety by blowing down a family of self-similar shrinking Ricci solitons.  More generally, the K\"ahler-Ricci flow with Calabi symmetry must develop Type-I singularity for any initial K\"ahler metric with Calabi symmetry \cite{Koi, Cao2, FF}. In fact, we conjecture that all smooth solutions of the K\"ahler-Ricci flow on K\"ahler surfaces can only develop Type-I singularities if it becomes singular in finite time. This leads to a more general speculation that the smooth K\"ahler-Ricci flow can only develop Type-I singularities for generic initial K\"ahler class.

The local models for contracting an exceptional divisor $\PP^{n-1}$ with normal bundle $\OO_{\PP^{n-1}}(-k)$ for $0<k<n$ is constructed in \cite{FIK} by solving nonlinear ODEs. Furthermore, it is shown in \cite{FIK} that there exists a  family of complete shrinking gradient K\"ahler-Ricci soliton $g(t)$ on the total space of $\OO_{\PP^{n-1}}(-k)$ for $t<0$ and $(\OO_{\PP^{n-1}}(-k), p, g(t))$ converges in pointed Gromov-Hausdorff topology for any $p$ in the exceptional divisor,  to a cone metric on $\CC^n/\mathbb{Z}_k$ and then can be uniquely extended in pointed Gromov-Hausdorff topology by a  family of complete expanding gradient K\"ahler-Ricci solitons $g(t)$ on $\CC^n/\mathbb{Z}_k$ for $t>0$. The above phenomena can also be generalized to the local flip model to understand the microlocal structure of a surgical metric flip by Ricci flow. The following proposition is essentially due to Li \cite{Li} by generalizing the results of \cite{FIK}, showing how the K\"ahler-Ricci flow should flow through flips locally.

\begin{proposition} \label{main4} Let $\EE^-= \OO_{\PP^m}(-1)^{\oplus(l+1)}$ and $\EE^+=\OO_{\PP^l}(-1)^{\oplus(m+1)}$ for $1\leq l < m$. Then $\EE^+$ is the flip of $\EE^-$ by  $$ \phi_-: \EE^- \rightarrow \hat \EE \leftarrow \EE^+: \phi_+  $$ with $\hat \EE$ being the affine cone over $\PP^m\times \PP^l$ in $\CC^{(m+1)(l+1)}$.  Then there exists a unique triple
$$\left\{ (\EE^-, g_-(t)), ~t\in (-\infty, 0);~ (\hat\EE, g_{\hat\EE} );~ (\EE^+, g_+(t)), ~t\in (0, \infty) \right\} $$
such that

\begin{enumerate}

\item $(\EE^-, g_-(t))$ is a smooth solution of the K\"ahler-Ricci flow induced by a complete shrinking gradient K\"ahler-Ricci soliton with Calabi symmetry;

\medskip

\item $(\EE^+, g_+(t))$ is a smooth solution of the K\"ahler-Ricci flow induced by a complete expanding gradient K\"ahler-Ricci soliton with Calabi symmetry;

\medskip

\item $(\hat\EE, g_{\hat\EE})$ is a sub-cone in $\CC^{(m+1)(l+1)}$ equipped with the cone metric $\ddbar (|z|^{2\lambda})$ for a unique $\lambda=\lambda(m, l)\in (0, 1)$;

\medskip

\item The triple is a smooth solution of the K\"ahler-Ricci flow on $$\EE^-\times (-\infty, 0) \cup \hat\EE \times\{ 0\} \cup \EE^+\times (0, \infty).$$ It performs the surgical metric flip at $t=0$ in pointed Gromov-Hausdorff topology by choosing a base point $p_-\in E^-$ for $\EE^-$ and $ p_+\in E^+$ for $\EE^+$, where $E^-$ and $E^+$ are the exceptional locus of $\phi_-$ and $\phi_+$.

\end{enumerate}

\end{proposition}

In section 6, we conjecture that for any flip $\phi_-: X^- \rightarrow Y \leftarrow X^+: \phi_+$ between two normal projective varieties with log terminal singularities, the flip can be realized by a continuation of a complete shrinking gradient soliton and a complete expanding gradient soliton through a tangent cone of $Y$. This can be viewed as a metric uniformization for algebraic singularities arising from divisorial contractions and flips.


\bigskip

\section{Local ansatz}

\subsection {Calabi ansatz}\label{2.1}  In this section, we will apply the Calabi ansatz introduced by Calabi \cite{C1} (also see \cite{Li, SY}) to understand the small contraction near the exceptional locus.

 Let $\EE= \OO_{\PP^m}(-1)\oplus \OO_{\PP^m}(-1) \oplus ... \oplus \OO_{\PP^m}(-1)=\OO_{\PP^m}(-1)^{\oplus (l+1)}$ be the holomorphic vector bundle over $\PP^m$ of rank $l+1$.  Let $z=(z_1, z_2, ..., z_m)$ be a fixed set of inhomogeneous coordinates for $\PP^m$ and
$$\theta= \ddbar \log (1+|z|^2) \in [\OO_{\PP^m}(1) ]$$
be the Fubini-Study metric on $\PP^m$ and  $h$ be the hermitian metric on $\OO_{\mathbb{P}^m} (-1)$ such that $\ric(h) = -\theta$. This induces a hermtian metric $h_{\EE}$ on $\EE$ is given by
 $$h_\EE= h^{\oplus (l+1)}. $$
Under a local trivialization of $\EE$, we write
$$e^\rho = h_\EE (z) |\xi|^2, ~~~ \xi = (\xi_0, \xi_1, ..., \xi_l ),$$ where $h_\EE(z)$ is a local representation for $h_\EE $ with
$h_\EE(z)= (1+|z|^2).$

Now we are going to  define a family of K\"ahler metrics on $\EE$ as below
\begin{equation}\label{metricrep1}
\omega = a \theta + \ddbar u(\rho)
\end{equation}
for an appropriate choice of convex smooth function $u= u(\rho)$ and $a>0$.  In fact, we have the following criterion due to Calabi \cite{C1} for the above form $\omega$ to be K\"ahler.
\begin{proposition}\label{kacon}

$\omega$ defined as above, extends to a global K\"ahler form on $\EE$ if and only if

\begin{enumerate}

\item[(a)]

$a>0$,

\item[(b)] $u'>0$ and $u''>0$ for $\rho\in (-\infty, \infty)$,

\item[(c)] $U_0 (e^\rho) = u(\rho)$ is smooth on $(-\infty, 0]$ with $U_0' (0)>0$.

\end{enumerate}

\end{proposition}

Straightforward calculations show  that
\begin{equation}\label{metricrep2}
\omega = (a + u'(\rho)) \theta + h_\EE e^{-\rho} ( u' \delta_{\alpha \beta} + h_\EE e^{-\rho} ( u'' - u') \xi^{\bar \alpha} \xi^{\beta} ) \nabla \xi^\alpha \wedge \nabla \xi^{\bar \beta}.
\end{equation}
Here, $$\nabla \xi^\alpha = d \xi^\alpha + h_\EE^{-1} \partial h_\EE \xi^\alpha$$ and $\{ dz^i, \nabla \xi^\alpha\}_{i=1, ..., m, \alpha=1, ..., l+1}$ is dual to the basis

\begin{equation*} \nabla_{z_i} = \frac{\partial}{\partial z_i} - h_\EE ^{-1} \frac{\partial h_\EE }{\partial z_i} \sum_{\alpha=0}^l \xi^\alpha \frac{\partial }{\partial \xi^\alpha}, ~~~~~~~ \frac{\partial}{\partial \xi ^\alpha}.
\end{equation*}

Let $L=\OO_{\PP^m}(-1)$. Then $\EE= L^{\oplus (l+1)}$. Let $p_\alpha: \EE \rightarrow L$ be the projection from $\EE$ to its $\alpha$th component, and let $e^{\rho_\alpha} = (1+|z|^2) |\xi_\alpha|^2$ and so $e^{\rho} = \sum_{\alpha=0}^l e^{\rho_\alpha}$.

\subsection{Algebraic local models. } From now on, we assume that $ 1 \leq l < m $. We let
$$\EE^- = \OO_{\PP^m}(-1)^{(l+1)} , ~  \EE^+ = \OO_{\PP^l}(-1)^{(m+1)} .$$   Let $E^-$ be the zero section of $\EE^-$, which is a projective space $\mathbb{P}^m$ with normal bundle $\mathcal{O}_{\mathbb{P}^m}(-1)^{\oplus (l+1)}$. We define $E^+$ as the zero section of $\EE^+$ similarly. By contracting $E^-$, one obtains the variety $\hat \EE$ with only one isolated point as singularity.  
Let $w=(w_1, w_2, ..., w_l)$ and $\eta=(\eta_0, \eta_1, ..., \eta_m)$ be the coordinates with a local trivialization defined on $\EE^-$ as in section 2.1. Then there exists a flip from $\EE^-$ to $\EE^+$
\begin{equation}
\begin{diagram}\label{diag1}
\node{\EE^-} \arrow{se,b,}{\phi_-}  \arrow[2]{e,t,..}{\check{\phi}}     \node[2]{\EE^+} \arrow{sw,r}{\phi_+} \\
\node[2]{\hat \EE}
\end{diagram},
\end{equation}
where $\hat\EE$ is an affine cone over $\PP^m\times \PP^l$ in $\CC^{(m+1)(l+1)}$ by the Segre embedding of
\begin{equation}\label{segre}
[Z_0, ..., Z_m]\times [W_0, ..., W_l] \in \PP^m\times \PP^l\rightarrow [Z_0W_0..., Z_i W_j ,..., Z_mW_l] \in \PP^{(m+1)(l+1)-1}.
\end{equation}
$\phi_-$ and $\phi_+$ are isomorphisms outside $E^-$ and $E^+$.
The birational flip $\check \phi$ can be viewed as change of coordinates as below
\begin{equation}\label{coorch1}
z_1= \frac{\eta_1}{\eta_0}, ~ z_2 = \frac{\eta_2}{\eta_0},~  ... ~, ~ z_m = \frac{\eta_m}{\eta_0};~ \xi_0 = \eta_0, ~\xi_1 = w_1 \eta_0, ~\xi_2 = w_2 \eta_0,~ ...~, ~ \xi_l= w_l \eta_0,
\end{equation}
or
\begin{equation}\label{coorch2}
w_1= \frac{\xi_1}{\xi_0}, ~ w_2 = \frac{\xi_2}{\xi_0}, ~...~ ,~ w_l = \frac{\xi_l}{\xi_0};~\eta_0 = \xi_0,~  \eta_0 = \xi_0, ~\eta_1 = z_1 \xi_0, ~ \eta_2 = z_2 \xi_0,~ ...~, ~ \eta_m= z_m \xi_0.
\end{equation}
We also have
\begin{equation}
\rho= \log (1+|z|^2)|\xi|^2 = \log (1+|w|^2) |\eta|^2.
\end{equation}

Let $\tilde \EE $  be the blow-up of $\EE^-$ along the zero section $E^-$.  $\tilde \EE$ can also be obtained by the blow-up of $\EE^+$ along $E^+$.  Then we have the following commutative diagram from section 1.9 in \cite{D}
\begin{equation}
\begin{diagram}
\node{\EE^- } \arrow{s,l}{{\small \pi_-} }     \node{\tilde \EE}  \arrow{w,t}{\small \vartheta_-} \arrow{s,r}{{\small \Psi}}  \arrow{e,t}{\small \vartheta_+}    \node{ \EE^+}  \arrow{s,r}{{\small \pi_+}}\\
\node{\mathbb{P}^m}      \node{\mathbb{P}^m\times \mathbb{P}^l} \arrow{w,t}{ p_-}   \arrow{e,t}{ p_+}\node{\mathbb{P}^l}
\end{diagram}
\end{equation}

For fixed $[\xi ]=[1, w]\in \mathbb{P}^{l}$, the proper transformation of $(\pi_+)^{-1} ([\xi])$ via $\check \phi ^{-1}$ is a smooth variety $L^-_{[\xi]}$ given by
\begin{equation}
L^-_{[\xi]}:  \xi = \xi_0 (1, w).
\end{equation}
Similarly, for fixed $[\eta]=[1, z] \in \mathbb{P}^m$, we define the smooth variety $L^+_{[\eta]}$ as the proper transformation of $(\pi_+)^{-1}$ via $\check \phi$ given by
$$ L^+_{[\eta]}: \eta = \eta_0 (1, z). $$

The relation between $L^-_{[\xi]}$ and $ L^+_{[\eta]}$ can be understood in the following lemma. In section 3, we will estimate the evolving metrics by the K\"ahler-Ricci flow for the holomorphic foliations $L^-_{[\xi]} $ and $ L^+_{[\eta]}$ parametrized by $[\xi]\in \PP^l$ and $[\eta]\in \PP^m$.

\begin{lemma} For any $[\xi]\in \PP^l$ and $[\eta]\in \PP^m$,  $L^-_{[\xi]}$ and $L^+_{[\eta]}$ are isomorphic to $\mathbb{C}^{m+1}$ blow-up at one point and $\mathbb{C}^{l+1}$ blow-up at one point respectively. In particular,
$$L^-_{[\xi]} \cap L^-_{[\xi']} = E^-, ~if~ [\xi]\neq [\xi'] \in \PP^l$$
and
$$L^+_{[\eta]}  \cap  L^+_{[\eta']} = E^+, ~if ~ [\eta]\neq [\eta']\in \PP^m.$$
Furthermore,
$$L^-_{[\xi]} \cap L^+_{[\eta]} \ = \mathbb{C}$$
and it is the line in $\mathbb{C}^{m+1}$ generated by a multiple of $\xi$ or the line in $\mathbb{C}^{l+1}$ generated by $\eta$.

\end{lemma}

\subsection{Analytic local forms} In this section, we will compare various local forms with Calabi symmetry. The pointwise comparison will be useful in proving estimates in section 3.  We keep the notations in section 2.2. Let
$$\theta_-= \ddbar \log (1+|z|^2), ~ \theta_+= \ddbar \log (1+|w|^2)$$ be the Fubini-Study metrics on $\mathbb{P}^m$ and $\mathbb{P}^l$.

We now fix a smooth closed $(1,1)$- form $\hat \omega$ on $ \EE^-$ by
\begin{eqnarray*}
\hat\omega
&=&  \ddbar e^{\rho} = \ddbar (1+|z|^2)|\xi|^2 \\
&= & \sqrt{-1} \left( |\xi|^2 dz_i \wedge d\bar z_i + \bar  z^i \xi^\alpha dz_i \wedge d\bar \xi_\alpha +    z^i  \bar \xi^\alpha d\xi_\alpha  \wedge d\bar z_i +   (1+|z|^2) d\xi_\alpha \wedge d\bar \xi_\alpha \right).
\end{eqnarray*}
$\hat\omega$ can also be viewed as a smooth closed $(1,1)$-form on $\EE^+$. $\hat\omega$ is strictly positive except at $E^-$ and $E^+$ and it lies in a big class on both $\EE^-$ and $\EE^+$.

We define
$$\omega_-= \theta_-+ \hat \omega, ~~ \omega_+ = \theta_+ + \hat \omega .$$
Obviously, they are smooth K\"ahler forms on $\EE^-$ and $\EE^+$ respectively.  We can now relate $\hat\omega$, $\theta_-$ and $\theta_+$ to K\"ahler metrics on $\tilde \EE$.

\begin{lemma} Let $\tilde \omega = \ddbar \rho + \hat \omega$. Then $\tilde \omega$ extends to a smooth K\"ahler form on $\tilde \EE$. In particular,
$$ \tilde \omega = \theta_- + \theta_+ + \hat\omega. $$

\end{lemma}

Let $\rho_\alpha =  (1+|z|^2) |\xi_\alpha|^2$. Then by letting $[\xi ]= [1, w]\in \PP^l$, we have
$$e^\rho = (1+|w|^2) e^{\rho_0} .$$
We choose a smooth closed nonnegative real $(1,1)$ form $\tau$ defined by
\begin{eqnarray*}
\tau
 &=& \ddbar \left( e^{\rho_0}  \right).
  \end{eqnarray*}
Although $\tau$ is not big, it defines a flat degenerate K\"ahler form on $L^-_{[\xi]}$ for each $[\xi]\in \PP^l $ as shown in the following lemma.

\begin{lemma}\label{locm1}

 Let $\nu=(\nu_1, \nu_2, ..., \nu_{m+1})$ be defined by
 $$\nu_1 =  \eta_1= \xi_0 z_1, ~ \nu_2 =\eta_2= \xi_0 z_2, ~ ...~ , ~\nu_m =\eta_n \xi_0 z_n , ~\nu_{m+1} =\eta_0= \xi_0 .$$
 Then $e^{\rho_0} = |\nu|^2$ and $$\tau =  \sqrt{-1} \left( d\nu_1 \wedge d\overline \nu_1 + d\nu_2\wedge d\overline \nu_2 + ... + d\nu_{m+1} \wedge  d \overline{\nu}_{m+1}\right)$$ is the pullback of the flat Euclidean metric on $\CC^{m+1}$.
Therefore
\begin{equation}
\hat \omega |_{L^-_{[\xi]}}= (1+|w|^2)\tau|_{L^-_{[\xi]}}
\end{equation}
 is a flat Euclidean metric flat on $L^-_{[\xi]} \setminus E^-(\simeq \CC^{m+1} \setminus \{0\})$ for each $[\xi]=[1,w] \in \PP^{l+1}$.

\end{lemma}

We are only interested the local behavior of these forms near the zero section $E^-$, so we define
\begin{equation}\label{defome}\Omega= \{ -\infty < \rho < 0\} \subset \EE^-  \setminus E^- = \EE^+ \setminus E^+,
\end{equation}
 or equivalently in local coordinates, $$e^\rho= (1+|z|^2)|\xi|^2 < 1. $$
Then for each $[\xi]=[\xi_0, ..., \xi_l]\in \PP^l$,
\begin{equation}\label{compp}
L^-_{[\xi]}  \cap \Omega= \{ (z_1, ... , z_m, \xi_0, \xi_2, ... , \xi_l)~|~ \xi_j= w_j \xi_0, ~e^{\rho_0} \leq (1+|w|^2)^{-1} \} .
\end{equation}
We now compare $\hat\omega$ to $\omega_-$, $\omega_+$ and $\tilde\omega$ on each holomorphic leaf $L^-_{[\xi]}\simeq \OO_{\PP^m}(-1)$.

\begin{lemma} \label{lcompa} There exists $C>0$ such that on $\Omega$

\begin{enumerate}

\item $\hat\omega \leq   \omega_- \leq C e^{-\rho}\hat\omega$,  $\hat\omega\leq \omega_+ \leq C e^{-\rho} \hat\omega$,

\medskip

\item   $  \tilde \omega|_{L^-_{[\xi]}}    =   \omega_-|_{L^-_{[\xi]}} $, for each $[\xi]\in \PP^l$,

\medskip

\item   $   \tilde \omega|_{L^+_{[\eta]}}   =    \omega_+ |_{L^+_{[\eta]}} $, for each $[\eta]\in \PP^m$.

\end{enumerate}

\end{lemma}

\begin{proof} The estimate (1) follows from straightforward calculation as $\omega_- = \hat\omega + \theta_-$. The equality (2) follows from $\tilde\omega|_{L^-_{[\xi]}} = \omega_-|_{L^-_{[\xi]}} + \theta_+|_{L^-_{[\xi]}} = \omega_-|_{L^-_{[\xi]}}$ since $\theta_+$ vanishes on $L^-_{[\xi]}$. (3) follows by the same argument.

\end{proof}

Finally, we make comparisons among the volume forms induced by $\hat \omega$, $\omega_-$, $\omega_+$ and $\tilde\omega$.

\begin{lemma}\label{volcompa} We use the coordinates in section 2.1. Then
\begin{eqnarray*}
\hat\omega^{m+l+1} &=& e^{m\rho} (1+|z|^2)^{-(m-l)} dz\wedge d\bar z \wedge d\xi\wedge d\bar \xi\\
(\omega_-) ^{m+l+1} &=& (1+e^\rho)^m (1+|z|^2)^{-(m-l)}  dz\wedge d\bar z \wedge d\xi\wedge d\bar \xi   \\
\tilde \omega^{m+l+1} &=& e^{-l\rho} (1+ e^\rho)^m (1+|z|^2)^{-(m-l)}   dz\wedge d\bar z \wedge d\xi\wedge d\bar \xi.
\end{eqnarray*}

\end{lemma}


\bigskip

\section{Small contraction by  Ricci flow}

In this section, we will prove a special case of Theorem \ref{main1} when only one $\PP^m$ with normal bundle $\OO_{\PP^m}(-1)^{n-m}$ is contracted by the K\"ahler-Ricci flow for $m \geq  n/2$. We will give the proof of Theorem \ref{main1} in section 5.2 in a more general framework of Mumford's quotients, while the main idea and technical arguments are contained in the proof of the following theorem.

\begin{theorem}\label{minus1}

Let $X$ be a projective manifold of $\dim_{\CC} X =n$ and let $g(t)$ be a smooth solution of the K\"ahler-Ricci flow for $t\in [0, T)$, starting from a smooth K\"ahler metric $g_0$ with $[g_0] \in H^{1,1}(X, \mathbb{R})\cap H^2(X, \mathbb{Q})$.  Suppose that

 \begin{enumerate}

 \item  there exists a complex submanifold $\mathbb{P}^m$ of $X$  with $ n/2 \leq m \leq n-2$ and normal bundle
 $$\OO_{\PP^{m}}(-1)^{\oplus (n-m)} ;$$

\item   the limiting K\"ahler class $\lim_{t\rightarrow T} [g(t)] $ is the pullback of an ample class on the projective variety $Y$ from the birational morphism $\pi: X \rightarrow Y$ by contracting $\PP^m$.

\end{enumerate}
Then the following holds.

\begin{enumerate}

\item $g(t)$ converges to a smooth K\"ahler metric $g(T)$ on $X\setminus \PP^m$ in $C^\infty(X\setminus \PP^m)$.

\medskip

\item The metric completion of $(X\setminus \PP^m, g(T))$ is a compact metric length space homeomorphic to $Y$. We denote it by $(Y, d_T)$

\medskip

\item $(X, g(t))$ converges in Gromov-Hausdorff topology to $(Y, d_T)$ as $t\rightarrow T$.

\end{enumerate}

\end{theorem}

The rest of  section 3 is devoted to proving Theorem \ref{minus1}. We first set up the complex Monge-Amp\`ere equation associated to the K\"ahler-Ricci flow on $X$. Let $\Theta$ be a smooth volume form on $X$ and let $\chi = \ric(\Omega) = \ddbar \log \Omega$ be the closed $(1,1)$ form in $[K_X]$. We define
$$\omega_t = \omega_0 + t \chi.$$
By the assumption in Theorem \ref{minus1}, $[\omega_t]$ is K\"ahler for all $t\in [0, T)$ and $[\omega_T]$ is the pullback of an ample class on $Y$. Without loss of generality, we can choose $\Theta$ such that $\omega_t$ is a Kaher form for $t\in [0, T)$ and $\omega_T$ is the pullback of the Fubini-Study metric of some embedding $\iota: Y \rightarrow \PP^N$ by a sufficiently large power of $[\omega_T]$.
Then the complex parabolic Monge-Amp\`ere equation associated to the K\"ahler-Ricci flow (\ref{krflow}) is given by
\begin{equation}\label{maeqn1}
\ddt{}\varphi = \log \frac{(\omega_t + \ddbar \varphi)^n}{\Theta}, ~~\varphi|_{t=0} = 0.
\end{equation}
The equation (\ref{maeqn1}) has a smooth solution $\varphi(t)\in PSH(X, \omega_t)$ for all $t\in [0, T)$.

The following proposition is well-known and is due to \cite{TZha, Zha, SoT1, SW1}. An elementary proof can also be found in \cite{SW2} (cf. Lemma 2.1).

\begin{proposition} Let $\varphi=\varphi(t)$ be the solution of (\ref{maeqn1}) and $\omega= \omega_t + \ddbar \varphi$. There exists $C>0$ such that  for all $t\in [0, T)$,

\begin{enumerate}

\item $ ||\varphi||_{L^\infty(X)} \leq C. $


\item $\omega^n \leq C \Theta$.

\item $\omega\geq C^{-1}\omega_T$.

\item As $t\rightarrow T$, $\varphi(t)$ converges pointwise on $X$ to $\varphi(T)\in PSH(X, \omega_T)\cap L^\infty(X)$. In particular, $\varphi(T)$ is constant on $\PP^m$ and $\omega$ converges in distribution sense to a K\"ahler current $\omega(T) = \omega_t + \ddbar \varphi(T)$ on $X$ as $t\rightarrow T$.

\end{enumerate}

\end{proposition}

Using Tsuji's trick (\cite{Ts, SoT1, TZha}) and the third order estimates (\cite{Y1, PSS}), one has the following immediate corollary.

\begin{corollary}\label{reduction1} With the assumptions in Theorem \ref{minus1}, for every compact set $K\subset X \setminus \PP^m$, there exist constants $C_{K, k}$ for $k=0, 1, ...,  $such that
$$||\omega||_{C^k(K)}\leq C_{K,k}, $$
where $C^k$ norm is taken with respect to a fixed K\"ahler metric on $X$. Therefore $\omega$ converges to a smooth K\"ahler metric $\omega (T)$ smoothly on $X\setminus \PP^m$ as $t\rightarrow T$.

\end{corollary}

Corollary \ref{reduction1} reduces the proof of Theorem \ref{minus1} to any sufficiently small neighborhood of the exceptional locus $\PP^m$. Without loss of generality, we can identify a small neighborhood of $\PP^m$ in $X$ by $\Omega$ defined as in (\ref{defome})  in the normal bundle $\EE^-$ of $\PP^m$.


\bigskip

\subsection{Second order estimates for holomorphic foliations }

In this section, we will prove a partial second order estimates. The original Yau's second order estimate will fail in the setting of the Theorem \ref{minus1} since the metric must blow up in the normal direction of $\PP^m$. Tsuji's trick \cite{Ts} can get a local second order estimates outside the exceptional locus, however, it does not produce a quantitive estimate near the singular set. One cannot directly apply the tricks in \cite{SW2, SW3} because the singular set $\PP^m$ has high codimension and the maximum principle does not work. Our idea is to estimate the evolving metrics of the K\"ahler-Ricci flow in a well-chosen  set of directions in the tangent space of each point on $X$ instead of all directions. We will use the holomorphic foliation constructed in section 2 (c.f. \cite{So3}) so that locally, one is able to reduce the high codimensional problem to a codimensional $1$ case. Here we exploit the algebraic and geometric structure of smooth flips studied in section 2.

We now can work on the normal bundle $\EE^-= \OO_{\PP^m}(-1)^{\oplus(l+1)}$ for $l=n-m-1$. We keep the same notations introduced in section 2. By choosing the coordinates $(z, \xi)$ appropriately, we can start doing estimates on $\Omega$ defined in (\ref{defome}) and we write $E^-\simeq \PP^m$ be the exceptional locus of the small contraction $\phi_-$.

\begin{definition} We define for $t\in [0, T)$ and $p\in \Omega$.
\begin{equation}
H (t, p )= \tr_{\hat\omega |_{L^-_{[\xi]}}} (\omega (t) |_{L^-_{[\xi]}} ) (p),
\end{equation}
where $p=(z, \xi)\in \EE^-.$

\end{definition}

Here $\hat\omega$ and $\omega(t)$ are restricted to $L^-_{[\xi]}$ as smooth real closed $(1,1)$-forms. $H$ can also be expressed as
$$ H(t, \cdot)  = \frac{ \omega(t) \wedge \hat\omega^m \wedge dw_1 \wedge d\bar w_1\wedge ... \wedge dw_l \wedge d\bar w_l }{\hat\omega^{m+1} \wedge dw_1 \wedge d\bar w_1\wedge ... \wedge dw_l \wedge d\bar w_l}. $$
We compare $\omega(t)$ to $\hat\omega$ in the fibre of the holomorphic foliation $L^-_{[\xi]}$ determined by the $\xi$ coordinates, or equivalently the fibre of $\phi_+$ over $[\xi]\in \PP^l$.

We first derive a lemma on the behavior of $\omega(t)$ near the boundary of  $\Omega$. This will allow us to apply the maximum principle.

\begin{lemma}\label{H1} $H \in C^\infty( \Omega)$  for all $t\in [0, T)$ and $p\in \Omega$ with
\begin{equation}\label{exH}
H(t, p) = (1+|w|^2)^{-1} \tr_{\tau |_{L^-_{[\xi]}}} (\omega(t)|_{L^-_{[\xi]}}),
\end{equation}
where $p=(z, \xi)\in \EE^-$. Furthermore,

\begin{enumerate}

\item for all $t\in [0, T)$, $$\sup_{\Omega} e^{\rho} H(t, \cdot) < \infty ;$$

\item  there exists $C>0$ such that for  all $t\in [0, T)$,

$$ \sup_{\partial \Omega\setminus E^-} e^{\rho} H (t, \cdot)  \leq C.$$

\end{enumerate}

\end{lemma}

\begin{proof}  The equality (\ref{exH}) follows directly by definition. Then (1) follows by (1) in Lemma \ref{lcompa}, and  (2) follows by Corollary \ref{reduction1}.

\end{proof}

 The following proposition is one of the key estimates.

\begin{proposition}\label{H2}  Let $\Box = \ddt{} - \Delta$ be the linearized operator of equation (\ref{maeqn1}), where $\Delta$ is the Laplace operator with respect to the evolving metric $\omega(t)$.  Then for all $t\in [0, T)$ and $p\in \Omega$, we have
\begin{equation}
\Box  \log H (t, p) \leq \tr_{\omega}(\theta_+),
\end{equation}
where $p=(z(p), \xi(p))$.

\end{proposition}

\begin{proof} We define $$I =(1+|w|^2)H = \tr_{\tau|L^-_{[\xi]}} (\omega|_{L^-_{[\xi]}}).$$  It suffices to show that
$$\Box \log I \leq 0$$
since $\theta_+=\ddbar \log (1+|w|^2) = \ddbar \log |\xi|^2.$
We break the proof into the following steps.

\bigskip

\noindent {\it \large  Step 1.} We first make a choice of special coordinates. On $\Omega$, we have the standard local coordinates with Calabi symmetry section 2, i.e., for each $p \in \Omega$,  we have at $p$, $ (z(p), \xi (p))$.  Once we fix $p$, there exists a unique $[1, w]=[\xi]\in \PP^l$ such that $p\in L^-_{[\xi]}$.

\begin{enumerate}

\item[(a)] Near $p \in \Omega$, we first choose the coordinates $(\nu, w)= (\nu_1, ..., \nu_{m+1}, w_1, ..., w_l)$, where $$\nu_1= \xi_0 z_1, ~...~, ~ \nu_m=\xi_0 z_m, ~\nu_{m+1} = \xi_0$$  as in Lemma \ref{locm1} . We will apply a linear transformation to $(\nu, w)$ such that  $$x=(x_1, x_2, ..., x_{m+l+1})^T= A^{-1} (\nu, w)^T.$$   Let $$x'=(x_1, ..., x_{m+1}), ~ x''= (x_{m+2}, ..., x_{m+l+1}). $$  We assume that $A$ is in the form of
    $$\left(
        \begin{array}{cc}
           A' & A'' \\
          0 & I_{l} \\
        \end{array}
      \right), $$ where $A'$ is an $(m+1)\times (m+1)$ matrix and $A''$ is an $(m+1) \times m$ vector. Immediately, we have $$x'' = w.$$

\item[(b)]  Suppose $g(t)$ at $(t, p)$ is given by the following hermitian matrix with respect to coordinates $(\nu, w)$

$$ G= \left(
        \begin{array}{cc}
          B' & B'' \\
          \overline{B''}^T & C \\
        \end{array}
      \right),$$ where $B'$ is an $(m+1)\times (m+1)$ hermitian matrix, $B''$ an $(m+1)\times l$ vector.   Then under the new coordinates $x$, $g(t)$ at $p$ is given by the following hermitian matrix

    \begin{eqnarray*}
    \bar A^T G A &=& \left(
                      \begin{array}{ccc}
                        \overline{A'}^T  B'  A' & & \overline{A'}^T B' A'' + \overline{A'}^T B'' \\
                        && \\
                        \overline{A''}^T B' A' + \overline{B''}^T A' & & \overline{A''}^T B' A'' + \overline{A''}^T B'' + \overline{B''}^T A'' + C  \\
                      \end{array}
                    \right) \\
                    &&\\
    &=& \left(
                      \begin{array}{ccc}
                        \overline{A'}^T  B'  A' & & \overline{A'}^T (B' A'' + B'') \\
                        &&\\
                        (\overline{A''}^T B' + \overline{B''}^T) A' &&  \overline{A''}^T B' A'' + 2Re(\overline{A''}^T B'' )+ C  \\
                      \end{array}
                    \right).
    \end{eqnarray*}

\item[(c)] We choose a unitary matix $A'$ such that $\overline{A'}^T B A'$ is diagonalized, i.e., $$ \left(
             \begin{array}{cccc}
               \lambda_1     & 0                  & ... & 0\\
               0                    & \lambda_2   & ...  & 0\\
                                     & ...                 & ...  &   \\
                              0     &  0                 & ...   & \lambda_{m+1}
             \end{array}
           \right) $$
     and choose $A''$ such that $$ B'A''= -B''$$ since $B'$ has rank $ m+1$. Therefore under the coordinates $x$, at $(t, p)$,

$$g= \left(
             \begin{array}{ccccc}
               \lambda_1     & 0                  & ... & 0                       & \\
               0                    & \lambda_2   & ...  & 0                      & O_{(m+1)\times l} \\
                                     & ...                 & ...  &                         & \\
                              0     &  0                 & ...   & \lambda_{m+1}&  \\
                                 &    O_{l\times (m+1)}             &  &               & -\overline{A''}^T B' A'' + C
             \end{array}
           \right) .
$$
The matrix representation of $\tau$ under the coordinates $(\nu, w)$ is given by  $$ \left(
             \begin{array}{ccccc}
     I_{m+1}     & O_{(m+1)\times l}            \\
    O_{l\times (m+1) }            &  O_{l \times l }
                 \end{array}
           \right) ., $$
and so  its matrix representation under the coordinates $x$ at $(t, p)$ is given by              $$\tau= \overline{A}^T  \left(
             \begin{array}{ccccc}
     I_{m+1}     & O_{(m+1)\times l}            \\
    O_{l\times (m+1) }            &  O_{l \times l }
                 \end{array}
           \right) A= \left(
                         \begin{array}{ccc}
                           I_{(m+1)\times (m+1)} & \overline{A'}^T A''  \\
                           \overline{A''}^T A' &  \overline{A''}^T A'' \\
                         \end{array}
                       \right)
$$ since $A'$ is unitary. Since $x''=w$ and on $L^-_{[\xi]}$, $x''$ is constant and $$\omega|_{L^-_{[\xi]}\cap \Omega} =\sqrt{-1} \sum_{i, j=1}^{m+1} g_{i\bar j} dx^i \wedge d\overline{x^j}.$$  Then at $(t, p)$, $$g|_{L^-_{[\xi]}} = diag( \lambda_1, ..., \lambda_{m+1}), ~~~~~~~\tau |_{L^-_{[\xi]}} = I_{m+1}$$
Finally, we arrive at
$$ I(t, p) = \sum_{i, j=1, ..., m+1}(\tau |_{L^-_{[\xi]}}) ^{i\bar j} (g|_{L^-_{[\xi]}})_{i\bar j} = \lambda_1 + \lambda_2+ ... + \lambda_{m+1} .$$

 \end{enumerate}

\noindent {\it \large Step 2.}   Now we calculate $\Box  ~I$ at $(t, p)$ under the coordinates $x$.  Notice that $\tau|_{L^-_{[\xi]}}$ is a constant form $\sqrt{-1} (dx_1\wedge d \overline x_1 + dx_2 \wedge d\overline x_2 + ... + dx_{m+1} \wedge d \bar x_{m+1} ) $ and so  all derivatives of $\tau|_{L^-_{[\xi]}}$ vanish.

We now apply the Laplace operator $\Delta$ to  $I$.
\begin{eqnarray*}
&&\Delta  \tr_{\tau|_{L^-_{[\xi]}}}(\omega|_{L^-_{[\xi]}}) \\
&=& \sum_{r, s=1}^{m+l+1}g^{r\bar s} \left( \sum_{i, j=1, ..., m+1} \left( \tau|_{L^-_{[\xi]}} \right)^{i\bar j} g_{i\bar j} \right)_{r\bar s}\\
&=&\sum_{r, s=1}^{m+l+1} \sum_{i, j=1, ..., m+1} g^{ r \bar s} \left( \tau|_{L^-_{[\xi]}} \right)^{i\bar j} g_{i\bar s, r \bar s} - \sum_{r, s=1}^{m+l+1}\sum_{i, j, p, q=1, ..., m+1}g^{r \bar s }g_{i\bar j} \left( \tau|_{L^-_{[\xi]}} \right)^{i\bar q} \left( \tau|_{L^-_{[\xi]}} \right)^{ p\bar j} \tau_{p\bar q, r\bar s} \\
&=& - \sum_{r, s=1}^{m+l+1} \sum_{i, j=1, ..., m+1} g^{r \bar s } \left( \tau|_{L^-_{[\xi]}} \right)^{i\bar j}R_{i\bar j r \bar s } + \sum_{r, s, p, q=1}^{m+l+1}\sum_{i, j =1, ..., m+1}g^{r \bar s} \left( \tau|_{L^-_{[\xi]}} \right)^{i\bar j} g^{p\bar q} g_{p \bar j, \bar s } g_{i\bar q, r}\\
&=& - \sum_{i, j=1, ..., m+1} \left( \tau|_{L^-_{[\xi]}} \right)^{i\bar j} R_{i\bar j} + \sum_{r , s, p,q=1}^{m+l+1}\sum_{i, j=1, ..., m+1}g^{r \bar s} \left( \tau|_{L^-_{[\xi]}} \right)^{i\bar j} g^{p\bar q} g_{p \bar j, \bar s} g_{i\bar q, r}.
\end{eqnarray*}
Then
\begin{eqnarray*}
&& \Box  \log I \\
&=& - ( I)^{-1}  \sum_{r, s=1}^{m+l+1} \sum_{i, j=1, ..., m+1; p, q=1, ..., m+l+1}g^{r\bar s}\left( \tau|_{L^-_{[\xi]}} \right)^{i\bar j} g^{p\bar q} g_{p \bar j, \bar s } g_{i\bar q, r} + (H )^{-2} |\nabla I |^2_g\\
&=&-  (I )^{-2}\sum_{r, s=1}^{m+l+1}\left\{ I   \sum_{i, j=1, ..., m+1; p, q=1, ..., m+l+1}g^{r\bar s} \left( \tau|_{L^-_{[\xi]}} \right)^{i\bar j} g^{p\bar q} g_{p \bar j, \bar s} g_{i\bar q, r}  \right.\\
&&~~~~\left. ~ -   g^{r\bar s}  \left(\sum_{i, j=1, ..., m+1} \left( \tau|_{L^-_{[\xi]}} \right)^{i\bar j} g_{i\bar j, r} \right)  \left(\sum_{i, j=1, ..., m+1}\left( \tau|_{L^-_{[\xi]}} \right)^{j\bar i} g_{j\bar i, \bar s} \right)    \right\}.
\end{eqnarray*}
%
\noindent {\it \large Step 3.}   The proof of the proposition is now reduced to show that
\begin{eqnarray*}  && \sum_{r, s=1}^{m+l+1} \left(  I  \sum_{i, j=1, ..., m+1; p, q=1, ..., m+l+1}g^{r\bar s} \left( \tau|_{L^-_{[\xi]}} \right)^{i\bar j} g^{p\bar q} g_{p \bar j, \bar s} g_{ i \bar q, r}  \right. \\
&&~ \left. ~ -   g^{r\bar s} \left(\sum_{i, j=1, ..., m+1} \left( \tau|_{L^-_{[\xi]}} \right)^{i\bar j} g_{i\bar j, r} \right)  \left( \sum_{i, j=1, ..., m+1} \left( \tau|_{L^-_{[\xi]}} \right)^{j\bar i} g_{j\bar i, \bar s} \right)    \right)\geq 0.
\end{eqnarray*}
We can further assume that $g$ is diagonalized at $(t, p)$ by applying a unitary matrix
$$ \left( \begin{array}{ccc}
I_{(m+1)\times (m+1)} & O_{(m+1)\times l} \\
O_{l  \times (m+1)} & D_{l\times l}
\end{array}
\right)
$$
with $D$ being a unitary $m\times m$ matrix such that $ \overline{D}^T ( C - \overline{A''}^T B' A'' ) D $ is diagonal.

Note that $\tau_{i\bar j}=\delta_{ij}$ for $i, j=1, ..., m+1$ and  $g= diag(\lambda_1, ... , \lambda_{m+l+1})$. Then
\begin{eqnarray*}
&&\sum_{r, s=1}^{m+l+1} g^{r\bar s} \left(\sum_{i, j=1, ..., m+1} \left( \tau|_{L^-_{[\xi]}} \right)^{i\bar j} g_{i\bar j, r} \right)  \left(\sum_{i, j=1, ..., m+1} \left( \tau|_{L^-_{[\xi]}} \right)^{j\bar i} g_{j\bar i, \bar s} \right)\\
&=&\sum_{r=1, ..., m+l+1} \lambda_r^{-1} \left|\sum_{i=1, ..., m+1}g_{i\bar i, r} \right|^2\\
&\leq& \sum_{i, j=1, ..., m+1} \left( \sum_{r=1, ..., m+l+1} \lambda_r^{-1} |g_{i\bar i, r}|^2\right)^{1/2}  \left(\sum_{r=1, ..., m+l+1} \lambda_r^{-1} |g_{j \bar j, r}|^2 \right)^{1/2}\\
&=&\left( \sum_{i=1, ..., m+1} \left(\sum_{r=1, ..., m+l+1} \lambda_r^{-1} |g_{i\bar i, r}|^2 \right)^{1/2}  \right)^2\\
&=&\left(   \sum_{i=1, ..., m+1} \lambda_i^{1/2} \left( \sum_{r=1, ..., m+l+1} \lambda_r^{-1}\lambda_i^{-1} |g_{i\bar i, r}|^2 \right)^{1/2}               \right)^2\\
&\leq& \left(\sum_{i=1, ..., m+1} \lambda_i \right) \left( \sum_{i=1, ..., m+1;r=1, ..., m+l+1} \lambda_r^{-1}\lambda_i^{-1} |g_{i\bar i, r}|^2 \right) \\
&\leq& I  \left( \sum_{r, s=1, ..., m+1,; i=1, ..., m+1} \lambda_r ^{-1}\lambda_l^{-1} |g_{i\bar s,  r}|^2 \right) \\
&=& I   \sum_{i, j=1, ..., m+1; r, s, p,q=1, ..., m+l+1}  g^{r\bar s} \left( \tau|_{L^-_{[\xi]}} \right)^{i\bar j}  g^{p\bar q} g_{p\bar j, \bar s} g_{i\bar q, r}  .
\end{eqnarray*}
This completes the proof of the proposition.

\end{proof}

\begin{corollary} \label{roue} There exists $C>0$ such for all  $p\in \Omega$ and  $t\in [0, T)$,
\begin{equation}
  \tr_{\hat\omega|_{L^-_{[\xi]}}} (\omega_{L^-_{[\xi]}} ) (t, p) \leq C e^{-\rho},
 \end{equation}
 where $p=(z(p), \xi(p))$.

\end{corollary}

\begin{proof}  Let $$J_{\epsilon} = \log \tr_{\hat\omega|_{L^-_{[\xi]}}} (\omega_{L^-_{[\xi]}} ) + (1+\epsilon) \rho$$ for $\epsilon>0$.  Then  for all $\epsilon\in (0, 1)$, $\limsup_{p \rightarrow E^-} J_\epsilon = -\infty$ by Lemma \ref{H1},  and on $\Omega$,
$$\Box J_\epsilon < 0,$$
because of Proposition \ref{H2} and the fact that
$$\Box \rho= - \Delta \log (1+|z|^2)|(1+|w|^2)|\xi_0|^2 = - \tr_\omega( \theta_- + \theta_+)>0 . $$

Applying the maximum principle for $J_\epsilon$,  the maximum of $J_\epsilon$ has to be achieved on $\partial \Omega$.  Then by Lemma \ref{H1}, there exists $C>0$ such that for all $\epsilon\in (0, 1)$ and $t\in[0,T)$,
$$\sup_{\Omega} J_\epsilon = \sup_{\partial \Omega\setminus E^-} J_\epsilon \leq \sup_{\partial \Omega \setminus E^-} I  \leq C.$$ The corollary is then proved by letting $\epsilon \rightarrow 0$.

\end{proof}

We define a holomorphic vector $V_-$ on $\Omega$ by
$$ V_-=   \xi_0 \frac{\partial}{\partial \xi_0} +  \xi_1 \frac{\partial}{\partial \xi_1}+... +  \xi_l \frac{\partial}{\partial \xi_l}.$$
$V_-$ vanishes along $E^-$ and
$$|V_-|^2_{\hat \omega} = |V_-|^2_{\omega_-}= e^\rho$$ by straightforward calculations. Direct calculations gives the following lemma.

\begin{lemma} $V_-$ is tangential to $L^-_{[\xi]}$ for each $[\xi] \in \PP^l$. The restriction of $V_-$ to $L^-_{[\xi]}$  is  a holomorphic vector field given by
\begin{equation}
V_-|_{L^-_{[\xi]}} =  \nu_1 \frac{\partial}{\partial \nu_1} +  \nu_2 \frac{\partial}{\partial \nu_2}+... +  \nu_{m+1} \frac{\partial}{\partial \nu_{m+1}}.
\end{equation}
In particular, $|\nu|^2(p) \leq e^\rho \leq 1$ when $p\in \Omega$.

\end{lemma}

We then consider the normalized vector field
$$W_-= \frac{V_-}{|V_-|_{\hat \omega}} = e^{-\rho/2} \sum_{j=0, ..., l} \xi_j \frac{\partial}{\partial \xi_j}.$$
The following proposition gives a useful estimate of the evolving metric in one of the normal direction of $E^-$. The proof is based on a similar idea in \cite{SW2}.

\begin{proposition} \label{keyest} There exists $C>0$ such that for all $t\in [0, T)$  and $p\in \Omega$,
\begin{equation} \label{tanest}
C^{-1}\hat\omega  \leq \omega(t) \leq Ce^{- l \rho} \hat \omega ,
\end{equation}
and
\begin{equation}\label{verest}
|W_-|^2_{\omega}   (t, p) \leq Ce^{-\rho/2},
\end{equation}
where $p=(z(p), \xi(p))$.

\end{proposition}

\begin{proof}  We break the proof into the following steps.

\bigskip

\noindent {\it \large Step 1.}   We apply the similar argument in the proof of Schwarz lemma \cite{Y2, SoT1}. Notice that $\omega (t) = \omega_t + \ddbar \varphi$ with $\varphi\in C^\infty(X)$ uniformly bounded in $L^\infty(X)$ for $t\in [0, T)$. Also there exists $C_1>0$ such that for all $t\in [0,T)$ and on $\Omega$, $\omega_t\geq   C_1 \hat\omega $ on $\Omega.$ Then we consider the quantity $$L = \log \tr_{\omega} (\hat\omega) -  2A \varphi.$$
$\hat\omega$ restricted to $\Omega$ is in fact the pullback of a flat metric on $\CC^{(m+1)(l+1)}$ given by a local morphism $(\xi_\alpha, z_i \xi_\beta)$ for $\alpha, \beta = 0, ..., l$ and $i=1, ..., m$. Then straightforward calculations give
$$\Box L \leq - A  \tr_{\omega}(\omega_t) +  2nA \leq  - A\tr_\omega(\hat\omega) +  2nA$$
for sufficiently large $A>0$.
Applying the maximum principle, we have
$$\tr_\omega(\hat\omega) \leq \sup_{\partial \Omega\setminus E^-} \tr_{\omega}(\hat\omega) +C$$
for a uniform constant $C>0$.
Note that $\tr_{\omega(t)}(\hat\omega) $ is uniformly bounded on $\partial \Omega$. Hence $\tr_{\omega}(\hat\omega)$ is uniformly bounded above and so there exists $C>0$ such that
\begin{equation}\label{sch1}
\omega\geq C \hat\omega.
\end{equation}

\bigskip

\noindent {\it \large Step 2.}   The maximum principle combined with Lemma \ref{volcompa} implies that  there exists $C >0$ such that
\begin{equation}\label{sch2}
 \omega^{m+l+1} \leq C e^{-l\rho} \hat\omega^{m+l+1}.
\end{equation}
 By the estimates (\ref{sch1}) and (\ref{sch2}), there exists $C >0$ such that
$$\tr_{\hat\omega} (\omega) \leq  Ce^{- l \rho} $$ and so $\omega\leq C e^{- l \rho} \hat\omega.$
This completes the proof for estimate (\ref{tanest}).

\bigskip

\noindent {\it \large Step 3.}   Let $V_0 = \xi_0 \frac{\partial}{\partial \xi_0}$ be the holomorphic vector field on $\Omega$. Then $V_0$ vanishes on $\rho_0=-\infty$ and
$ |V_0|^2_{\hat \omega} = (1+|w|^2) |\xi_0|^2= e^{\rho_0}.$ In particular,
$$|V_-|^2_\omega = (1+|w|^2) |V_0|^2_\omega. $$
Using the normal coordinates for $\omega$, we can show that
$$\Box |V_0|^2_g = \ddt{}( |V_0|^2_g) + (V_0)^i (V_0)^{\bar j} R_{i\bar j} - g^{r\bar s} g_{i\bar j} \left((V_0)^{i} \right)_r \left((V_0)^{\bar j} \right)_{\bar s} = -|\partial V_0|^2_g$$
and so
$$\Box \log |V_0|^2_g = -(|V_0|^2_g)^{-2} \left( |V_0|^2_g |\partial V_0|^2_g - |\nabla_t |V_0|_g^2|^2  \right)\leq 0. $$
We now define for $\epsilon\in (0, 1)$, $t\in [0, T)$ and $p\in \Omega$,
$$G_\epsilon = \log \left( e^{\epsilon \rho_0}|V_0|^2_\omega \tr_{\tau|_{L^-_{[\xi]}}} (\omega|_{L^-_{[\xi]}}) \right), $$
where $p=(z(p), \xi(p))$. Then
$$G_\epsilon= \log I + \log |V_0|_\omega^2 + \epsilon \rho_0= \log |V_-|^2_\omega + \log \tr_{\hat\omega|_{L^-_{[\xi]}}} (\omega|_{L^-_{[\xi]}}) +\epsilon \rho_0 .$$
For each $t\in [0, T)$, $G_\epsilon$ is smooth in $\Omega$,  and it tends to $-\infty$ near $\rho_0=-\infty $  for all $\epsilon \in [0, T)$ by Lemma \ref{H1}. Furthermore, there exists $C>0$ such that for all $\epsilon\in (0,1)$ and $t\in [0, T)$,
$$ \sup_{\partial \Omega} G_\epsilon  \leq C. $$
On the other hand,
$$ \Box G_\epsilon= \Box \log I + \Box \log |V_0|^2_\omega + \epsilon \Box \rho_0 <  0.$$
By the maximum principle, $$\sup_{\Omega} G_\epsilon \leq \sup_{\partial \Omega } G_\epsilon \leq C.$$
Then by letting $\epsilon$ tend to $0$, for each  $p\in \Omega$ and $t\in [0, T)$ we have
\begin{equation}\label{dirc}
|V_-|^2_\omega \tr_{\hat\omega|_{L^-_{[\xi]}}} (\omega|_{L^-_{[\xi]}}) \leq C.
\end{equation}

\bigskip

\noindent {\it \large Step 4.}   Note that $V_-$ is tangential to each $L^-_{[\xi]}$ for $[\xi]\in \PP^l$.
$$ | V_-|_{L^-_{[\xi]}} |^2= |V_-|^2$$
and there exists $C>0$ such that for each $p\in \Omega$ and $t\in [0, T)$,
$$|V_-|^4_{\omega}  \leq  \left|V_- |_{L^-_{[\xi]}} \right|^2_{\omega}  |V_-|^2_\omega \leq \left|V_- |_{L^-_{[\xi]}} \right|^2_{\hat\omega} |V_-|^2_\omega \tr_{\hat\omega|_{L^-_{[\xi]}}} (\omega|_{L^-_{[\xi]}})  \leq C e^\rho .$$
Therefore there exists $C>0$ such that
$$ |W_-|^2_\omega \leq C e^{-\rho/2}$$ uniformly for $t\in [0, T)$ and $p\in \Omega$. This completes the proof of the proposition.

\end{proof}

Similarly, we can define a holomorphic vector $V_+$ on $\Omega$ by
$$ V_+=   \eta_0 \frac{\partial}{\partial \eta_0} +  \eta_1 \frac{\partial}{\partial \eta_1}+... +  \eta_m \frac{\partial}{\partial \eta_m}$$
using the coordinates $\eta$ defined in (\ref{coorch2}).  $V_+$ vanishes on $E^+$ and
$$|V_+|^2_{\hat \omega} = |V_+|^2_{\omega_+}= e^\rho.$$ We also consider the normalized vector field
$$W_+= \frac{V_+}{|V_+|_{\hat \omega}} = e^{-\rho/2} \sum_{j=0}^{ m} \eta_j \frac{\partial}{\partial \eta_j}.$$

We now derive the main estimate in this section.

\begin{proposition} \label{keyest2}  Let $\omega(t)$ be the solution of the K\"ahler-Ricci flow on $X$ for $t\in [0, T)$. There exist $\gamma>0$ and  $C>0$ such that for all $t\in [0, T)$ and   $p\in \Omega$,
\begin{equation} \label{rouest}
e^\rho \tr_{\hat\omega|_{L^-_{[\xi]}}}(\omega|_{L^-_{[\xi]}}) (t, p) \leq C,
\end{equation}
\begin{equation}\label{analo1}
 e^\rho \tr_{\hat \omega|_{L^+_{[\eta]}}}(\omega|_{L^+_{[\eta]}})   (t, p) \leq C,
\end{equation}
\begin{equation}  \label{analo2}
|W_-|^2_{\omega}   (t, p) + |W_+|^2_{\omega}  (t, p) \leq Ce^{-\rho/2},
\end{equation}
\begin{equation}  \label{2ndor}
e^{2\rho} \left( \tr_{\tilde \omega}(\omega) \right)^\gamma  \tr_{\hat\omega|_{L^-_{[\xi]}}}(\omega|_{L^-_{[\xi]}}) \tr_{\hat \omega|_{L^+_{[\eta]}}}(\omega|_{L^+_{[\eta]}})   (t, p) \leq C,
\end{equation}
where $\xi= \xi(p)$ and $\eta=\eta(p)$ are defined in (\ref{coorch1}) and (\ref{coorch2}).

\end{proposition}

\begin{proof}  The proof of (\ref{analo1}) (\ref{analo2}) is the same as that of  Proposition \ref{keyest}. It suffices to prove (\ref{2ndor}). Note that $\omega$ is uniformly bounded on $\partial \Omega\setminus E^-$. All the quantities involved are uniformly bounded on $(\partial \Omega)\setminus \{\rho=0\}$. From the previous estimates, we have

$$\ddbar  \rho = \ddbar  \log (1+|z|^2)(1+|w|^2) |\xi_0|^2 \geq \theta_- + \theta_+, $$

$$\Box \log e^\rho \tr_{\hat\omega|_{L^-_{[\xi]}} }(\omega|_{L^-_{[\xi]}}) \leq  -\Delta \rho + \tr_{\omega}(\theta_+) \leq - \tr_{\omega}(\theta_-) .$$
Same calculations for $L^-_{[\xi]}$ show that
$$\Box \log e^\rho \tr_{\hat\omega|_{L^+_{[\eta]}} }(\omega|_{L^+_{[\eta]}}) \leq   - \tr_{\omega}(\theta_+), $$
and so the estimates (\ref{analo1}) and (\ref{analo2}) follow by the same argument before.

Let $H_\gamma = \log e^{2\rho} \left( \tr_{\omega_+}(\omega) \right)^\gamma  \tr_{\hat\omega|_{L^-_{[\xi]}}}(\omega|_{L^-_{[\xi]}}) \tr_{\hat \omega|_{L^+_{[\eta]}}}(\omega|_{L^+_{[\eta]}}) - \varphi $.
Standard calculations show that there exists $A>0$ such that $$\Box \log \tr_{\omega_+ }(\omega) \leq A \tr_{\omega}(\omega_+) + A $$ uniformly on $\Omega$ and for $t\in [0, T)$.
Then
$$\Box H_\gamma \leq - (1- A\gamma) \tr_{\omega} ( \omega_+)  + C. $$
The estimate (\ref{2ndor}) follows by choosing $\gamma>0 $ sufficiently small.

\end{proof}

The following corollary shows that the restriction of $\omega(t)$ to the exceptional rational curve is uniformly bounded above.
\begin{corollary}\label{corestE}
There exists $C>0$ such that for all $t\in [0, T)$ and $p\in E^-$,
\begin{equation}\label{estE}
\omega(t)|_{E^-} \leq C \theta_-|_{E^-}.
\end{equation}

\end{corollary}

\begin{proof}

For any point $p\in E^-$,  $p\in L^-_{[\xi]}$ for each $[\xi]\in \PP^l$ and  there exist $e_1, ..., e_m  \in T_p E^-$ and $e_{m+1} \in T_p L^-_{[\xi]}$  such that they form an orthonormal basis of $T_p L^-_{[\xi]}$ with respect to $ \omega|_{L^-_{[\xi]}}$. We can further assume that $\theta_-$ is diagonalized at $p$ in $T_p E^-$. Then at $p$
\begin{eqnarray*}
&&\tr_{\theta_-|_{E^-}} (\omega|_{E^-}) \\
&=&  \left. \frac{\omega\wedge (\theta_-)^{m-1}(e_1\wedge\overline{e_1}\wedge...\wedge e_m\wedge\overline{e_m}) }{ (\theta_-)^m(e_1\wedge\overline{e_1}\wedge...\wedge e_m\wedge\overline{e_m})} \right|_{E^-} \\
&=&  \left. \frac{ ( \omega\wedge (\theta_-)^{m-1} ) (e_1\wedge \overline{e_1}\wedge...\wedge e_m \wedge\overline{e_m}) ~ \hat\omega(e_{m+1}\wedge \overline{e_{m+1}})}{ (\theta_-)^m (e_1\wedge \overline{e_1} \wedge ... \wedge e_m \wedge\overline{e_m}) ~\hat\omega(e_{m+1}\wedge \overline{e_{m+1}})} \right|_{L^-_{[\xi]}}\\
&& +  \left. \frac{\omega(e_{m+1}\wedge\overline{e_{m+1}}) ~  \left( (\theta_-)^{m-1}\wedge \hat\omega \right)  (e_1\wedge\overline{e_1}\wedge...\wedge e_m \wedge\overline{e_m} )}{ (\theta_-)^m (e_1\wedge \overline{e_1}\wedge...\wedge e_m \wedge\overline{e_m}) ~\hat\omega(e_{m+1}\wedge \overline{e_{m+1}})} \right|_{L^-_{[\xi]}}  \\
&\leq&  \left. \frac{\omega\wedge (\theta_-)^{m-1}\wedge \hat\omega }{ (\theta_-)^m \wedge \hat\omega} \right|_{L^-_{[\xi]}}
\end{eqnarray*}
since $\hat\omega$ vanishes on $E^-$. By  in Corollary \ref{roue}, there exist $C_1, C_2>0$ such that for all $t\in [0, T) $ and $p\in \Omega$,
$$ \tr_{\theta_-|_{E^-}} (\omega|_{E^-}) \leq  \left.  \frac{\omega \wedge \theta^{m-1} \wedge \hat\omega}{ (\theta_-)^m \wedge \hat\omega} \right|_{L^-_{[\xi]}} \leq   C_1 \left.   e^{-\rho}  \frac{ (\theta_-)^{m-1} \wedge \hat\omega^2} { (\theta_-)^m \wedge \hat\omega} \right|_{L^-_{[\xi]}}  \leq C_2 $$
\end{proof}

\subsection{Proof of Theorem \ref{minus1} }

In this section, we will develop geometric estimates using the analytic estimates from section 3.1 and finish the proof of Theorem \ref{minus1}.

\begin{proposition} \label{onE} Let $g(t)$ be the solution of the K\"ahler-Ricci flow on $X$ for $t\in [0,T)$. There exists $C>0$ such that for all $t\in [0, T)$

\begin{enumerate}

\item $diam(X, g(t)) \leq C. $

\medskip

\item $ diam(E^-, g(t) |_{E^-}) \leq C (T-t)^{1/3}. $

\end{enumerate}

\end{proposition}

\begin{proof}  For (1), it suffices to show that the diameter of $\Omega$ is uniformly bounded. For any point $p\in \Omega$ with coordinates $(z, \xi)$, there exist $q\in E^-$ and a smooth path $\gamma_{p,q}$ in $L^-_{[\xi]}\cup \overline{\Omega}$ such that the length of $\gamma_{p, q}$ with respect to $g(t)$ is uniformly bounded, using (\ref{rouest}) in Proposition \ref{keyest2} to compare the evolving metrics $g(t)$ to the Euclidean metric $\hat\omega$ in $\mathbb{C}^{m+1}$. On the other hand, from Corollary \ref{estE}, the diameter of $E^-$ and so the diameter of $\partial \Omega$ in $\EE^-$  with respect to $g(t)$ are both uniformly bounded with respect to $g(t)$. This immediately implies that the diameter of $\overline{\Omega}$ in $\EE^-$ with respect to $g(t)$ is uniformly bounded.

To prove (2), we first note that  there exists $C>0$ such that for  any projective line $l$ in $E^-$
\begin{equation}\label{estE2}
\int_l \omega(t)   \leq C (T-t)
\end{equation}
by calculating the intersection number  $[\omega(t)]\cdot l$. Then we can apply the same argument in \cite{SW2} to prove (2) by combining (\ref{estE}) and (\ref{estE2}).

\end{proof}

The following theorem is an adaption from Theorem 5.1 in \cite{SW3}.

\begin{proposition} \label{geomest} Let $\mathcal{A}_{r, r'} = \{    r' \leq    e^{\rho/2} \leq  r \} $ for $0<r'<r\leq 1$. There exist $\kappa > 2$ and $C>0$  such that for $0<r\leq 1/\kappa$

\begin{enumerate}

\item for  any two points $p, p'\in \mathcal{A}_{r, 2r} $, there exists a  piecewise smooth path $\gamma \in \mathcal{A}_{  r/\kappa , \kappa r}$ joining $p$ and $p'$ satisfying,

$$L_{g(t)} (\gamma) \leq C r^{1/6}$$
whenever  $t\in [0, T)$ with $0< T-t \leq r^{1/2}$;

\item  For  any two points $p, p'\in \mathcal{A}_{r, 2r}$, there exists a  piecewise smooth path $\gamma\in \mathcal{A}_{r/\kappa, \kappa r}$ joining $p$ and $p'$ satisfying

$$L_{g(T)} (\gamma) \leq C r^{1/6}. $$

\end{enumerate}

\end{proposition}

\begin{proof} First we note that  $\AAA_{r/\kappa, \kappa r} \subset \Omega$.  For any two points $p, q\in \AAA_{r/2, 2r}$, we write $p=(z, \xi)= (z, \lambda (1, w))$ and $q=(z', \xi')=(\lambda' (1, w'))$ and let $[1, w]= [\xi]\in \mathbb{P}^l$ and $[1, w'] = [\xi']\in \mathbb{P}^l$. Without loss of generality, we can assume that $|z|, |z'|, |w|, |w'| \leq 1$. This is because we can find points $q_1=p, q_2, ..., q_k=q \in \AAA_{r/2,2r}$ for $k$ independent on $r$ and $t$ such that for each pair $q_i$ and $q_{i+1}$, one can find a coordinate chart on an inhomogeneous coordinate system for $[\xi ]=[1, w]$ so that $|z|, |w|\leq 1$ at $p_i$ and $p_{i+1}$.   We then choose an intermediate point $o$ with $o=(z, \lambda'' \xi ' ) $ and
$$(1+|z|^2) | \lambda'' \xi' /\lambda'|^2 = r^2. $$
Obviously, $o\in \AAA_{r/2, 2r}$ and
$$o, p \in L^+_{[1, z]} \cap \AAA_{r/2, 2r}, ~ o, q \in L^-_{[\xi']} \cap \AAA_{r/2, 2r} .$$
$L^+_{[1, z]}$ and $L^-_{[1, w']}$ are isomorphic to $\mathbb{C}^{l+1}$ and $\mathbb{C}^{m+1}$ blow-up at the origin respectively. Furthermore,
$$ r^2/8 \leq  |\lambda|^2 \leq 2 r^2, ~r^2/8\leq  |\lambda'|^2 \leq 2 r^2, ~r^2/8 \leq  |\lambda''|^2 \leq 2 r^2.  $$
Then we can apply Theorem 5.1 in \cite{SW3} and conclude that there  exist $C>0$ and piecewise smooth paths $\gamma_{o,p}$ and $\gamma_{o,q}$ joining $o, p$ in $L^+_{[1,z]} \cap \AAA_{r/\kappa, \kappa r}$  and $o, q$ in $L^-_{[1, w']} \cap \AAA_{r/\kappa, \kappa r}$ respectively, such that for all $t$ with $0< T-t \leq r^{1/2}$,  the arc lengths are bounded by
$$L_{g(t)} (\gamma_{o,p}) \leq C r^{1/6}, ~ L_{g(t)} (\gamma_{o,q}) \leq C r^{1/6}. $$

This immediately implies that there exists a piecewise smooth path joining $p$ and $q$ in $\AAA_{r/\kappa, \kappa r}$ with length uniformly bounded by $3Cr^{1/6}$ for some fixed $\kappa\geq 2$. This proves (1).

The estimate (2) follows from the same argument in the proof of (1) by letting $t\rightarrow T$ because for fix $r>0$, $g(t)$ converges uniformly in $C^\infty$ on $\AAA_{r/\kappa, \kappa r}$.

\end{proof}

The following corollary immediately follows from Proposition \ref{geomest} and it can be used to show that the metric completion of $(X\setminus E^-)$ is homeomorphic to $Y$ as the exceptional locus $E^-$ is contracted to point as $t\rightarrow T$.

\begin{corollary} \label{metrcom1}

Let $g(T)$ be the limiting smooth K\"ahler metric on $\Omega$ of the K\"ahler-Ricci flow as $t\rightarrow T$. For any $\epsilon>0$ there exists $r_0 \in (0, 1)$ such that
\begin{equation}
\emph{diam}_{g(T)} (B_r \setminus  \{ E^-\}) < \epsilon, \quad \textrm{for all } 0< r \leq r_0,
\end{equation}
where $B_r = \{ p \in \Omega~|~ e^\rho \leq r \}$.

\end{corollary}

Now we can complete the proof  Theorem \ref{minus1}. First of all, by the smooth convergence of $g(t)$ to $g(T)$ on $X\setminus E^-$ and Corollary \ref{metrcom1}, the metric completion of $(X\setminus E^-, g(T))$ is homeomorphic to the blow-down variety $Y$. It is then straightforward to verify that $(X, g(t))$ converges to $(Y, d_Y)$, the metric completion of $(X\setminus E^-, g(T))$ because by Proposition \ref{geomest} and the smooth convergence of $g(t)$ on $X\setminus E^-$, for any $\epsilon>0$, there exist $\delta>0$ and $\gamma>0$  such that  the set $\{ p\in X~|~ e^\rho\geq \delta \}$ is an $\epsilon$-dense set both in $(X, g(t))$ and $(Y, d_T)$ for $0<T-t<\gamma$.


\bigskip

\section{Small resolution by  Ricci flow}

In this section, we give two approaches to resolve a family of isolated singularities of a normal projective variety using Ricci curvature. It was first shown in  \cite{SY}  for a very special family of varieties with K\"ahler metrics satisfying the Calabi symmetry. It turns out the Ricci flow approach seems to be much more canonical because it gives a unique minimal resolution.

Assume that $Y$ is an $n$-dimensional normal projective variety with isolated singularities $p_1, ..., p_k$. Suppose near each $p_j$, $Y$ is isomorphic to the complex cone over $\PP^{n_j} \times \PP^{n-n_j-1}$ in $\mathbb{C}^{(n_j+1)\times(n-n_j)}$ by the Segre embedding from $\PP^{n_j} \times \PP^{n-n_j-1} \rightarrow \PP^{(n_j+1)\times(n-n_j)-1}$
$$ [Z_0, ..., Z_{n_j} ] \times [W_0, ...,  W_{n-n_j-1} ] \rightarrow [Z_0W_0 ..., Z_i W_j, ...Z_{n_j} W_{n-n_j-1}]$$
with $n_j, n-n_j \geq 2$ and $n_j \neq n-n_j$.  Without loss of generality, we can assume $n_j \geq  n-n_j$.

There exist two minimal resolutions
$$\phi_- : X^- \rightarrow Y$$
 over $p_1, ..., p_k$ such that the exceptional locus $E^-_j = (\phi_-)^{-1} (p_j)$ is isomorphic to $\mathbb{P}^{n_j}$ with normal bundle $\OO_{\PP^{n_j}}(-1)^{\oplus (n-n_j)}$
and
$$\phi_+: X^+ \rightarrow Y$$
  over $p_1, ..., p_k$ such that the exceptional locus $E^+_j = (\phi_+)^{-1} (p_j)$ is isomorphic to $\mathbb{P}^{n-n_j-1}$ with normal bundle $\OO_{\PP^{n-n_j-1}}(-1)^{\oplus (n_j+1)}$.  $X^+$ is exactly the flip of $X^-$ over $Y$.

There also exists a non-minimal resolution of $Y$ over $p_1, ..., p_k$ given by  $$\tilde\phi: \tilde X \rightarrow Y$$ with $\tilde X$ being the blow-up of $X^-$ and $X^+$ along $E^-$ and $E^+$ and $\tilde{\phi}^{-1} (p_j)$ is isomorphic to the exceptional locus $\tilde E_j = \PP^{n_j}\times \PP^{n-n_j-1}$.

\subsection{Resolution of singularities by Ricci curvature}

We keep the same assumptions and the notations as before. Then we have the following theorem generalizing Proposition 5.1  in \cite{SY}.

\begin{theorem} Let $g$ be a K\"ahler metric  $Y$ such that $g$ is smooth on $Y\setminus\{ p_1, ..., p_k\}$ and near $p_j$, $g$ is the restriction of a smooth K\"ahler metric on $\CC^{(n_j+1)(n-n_j)}$ by a local embedding induced from the Segre map.
There exists $\epsilon_0>0$ such that for any $\epsilon\in (0, \epsilon_0)$, $g- \epsilon \ric(g)$ is a smooth K\"ahler metric on $Y\setminus \{ p_1, ..., p_k \}$. Furthermore, there exists a smooth K\"ahler metric $\tilde g_\epsilon $ on $\tilde X$ such that  $(\tilde X, \tilde g_\epsilon)$ is  the metric completion of $(Y\setminus \{p_1, ..., p_j \}, g-\epsilon  \ric(g))$.

\end{theorem}

\begin{proof}

Both $g$ and $\ric(g)$ are smooth on $Y\setminus \{p_1, ..., p_k\}$. So it suffices to prove the theorem in a neighborhood of $p_j$. Without loss of genreality, we can assume that $Y$ has only one isolated singularity $p$. Let $X^+$ be a nonsingular model $\phi_+: X^+ \rightarrow Y$ of $Y$. Suppose $(\phi_+)^{-1} (p) = \mathbb{P}^l$ with $m=n-l-1$ and $m\geq n/2$, whose normal bundle is $$ \mathcal{O}_{\mathbb{P}^l}(-1)^{\oplus(m+1)}.$$  We now work locally near $E^+$ or simply a neighborhood of the zero section of the normal bundle $\EE^+= \OO_{\PP^l}(-1)^{\oplus(m+1)}$.

We keep the same coordinates and notations as in section 2,
$$e^\rho = (1+|z|^2) |\xi|^2. $$

Using the local Segre embedding of $Y$ near $p$ into $\CC^{(m+1)(l+1)}$,  $\hat\omega= \ddbar e^\rho$ on $Y$ is exactly the restriction of the Euclidean metric on $\CC^{(m+1)(l+1)}$.

The volume form $\omega^n$ on $Y$ is equivalent to the restriction of $\hat\omega^n = (\ddbar e^\rho)^n$. More precisely,  there exist a real-valued function $F$ on $Y$ and  $C>0$ such that
$$\omega^n = e^F \hat\omega^n, ~  -C\leq F\leq C $$
since $\omega$ and $\hat\omega$ are $L^\infty$ equivalent on $\CC^{(m+1)(n-m)}$ and their pullback are both smooth on $X^+$. In particular, $(\phi_+)^{-1} (F)$ is a smooth function on $ X^+$ by comparing $\hat\omega^n$ and $\omega^n$ using their Taylor expansions in $\xi_j$ and $z_i\xi_j$. Both $\omega^n$ and $(\hat\omega)^n$ have  the same vanishing order $e^\rho$, in addition, $F$ is constant on $E^+$ and $\ddbar F$ is bounded below and above by a multiple of $\hat\omega+\theta_+$. Since the pullback $F$ is  also a smooth function on $X^-$ with $\ddbar F$ bounded below and above by a multiple of $\hat\omega+\theta_-$, we conclude that there exists $A>0$ such that
$$-A\hat\omega \leq \ddbar F \leq A \hat\omega. $$
Then on $\tilde X\setminus \tilde E$, we have $$\ric(\omega ) = \ric(\hat\omega) -\ddbar F = -m \theta_-  -l \theta_+ - \ddbar F. $$
and
$$ \omega - \epsilon \ric(\omega ) = \omega + \epsilon (m \theta_- + l  \theta_+ -\ddbar F). $$
For $\epsilon>0$ sufficiently small,  $\omega -\epsilon \ric(\omega)$ is K\"ahler on $\tilde X\setminus \tilde E$ and smoothly extends to a smooth form on $\tilde X$. By comparing it to the K\"ahler form $\hat\omega+ \theta_-+\theta_+$ on $\tilde X$, $\omega -\epsilon \ric(\omega)$ extends to a global K\"ahler form on $\tilde X$. The theorem immediately follows.

\end{proof}

\subsection{Resolution of singularities by Ricci flow} In this section, we are going to prove the following theorem as a special case of Theorem \ref{main2}.

\begin{theorem} \label{minus2} Let $Y$ be an $n$-dimensional  projective variety of $\dim_{\CC} Y=n$ with an isolated singularity $p$. Suppose

\begin{enumerate}

\item $\pi: X \rightarrow Y$ is a resolution of singularities at $p$ with $$\pi^{-1} (p) = \mathbb{P}^l, ~1\leq l \leq n-2. $$

\item the normal bundle of $\mathbb{P}^l$ is $\mathcal{O}_{\mathbb{P}^l}(-1)^{\oplus(n-l)}$ with $ l<n/2 . $

\end{enumerate}
Let $g_0$ be a smooth K\"ahler metric on $Y$, i.e., $g_0$ is the restriction of a smooth K\"ahler metric on any local embedding of $Y$ in $\CC^N$ for some $N$. Then there exists a unique smooth solution $g(t)$ of the K\"ahler-Ricci flow on $X$ for $t\in (0, T)$ for some $T\in (0, \infty]$ satisfying

\begin{enumerate}

\item $g(t)$ converges to  $g_0$ on $X\setminus \PP^l$ in $C^\infty(X\setminus \PP^l)$.
\medskip

\item $(X, g(t))$ converges in Gromov-Hausdorff topology to $(Y, g_0)$ as $t\rightarrow 0$.

\end{enumerate}

\end{theorem}

Since $\omega_0$ is locally the pullback of a smooth metric on $\CC^N$ for a local embedding of $Y$, $\omega_0$ is a smooth closed big and nonnegative $(1,1)$-form on $X$. We define
$$T= \sup\{ t>0 ~|~  [\omega_0] + t[K_X] ~\textnormal{is~K\"ahler~on ~} X\}. $$
Such $T\in (0, \infty]$ is well defined because $K_X$ is ample on each component of the exceptional locus  of $\pi$. By choosing a suitable smooth closed $(1,1)$-form in $ [K_X]$, we can assume $\omega_t = \omega_0 + t \chi$ is a K\"ahler metric for each $t\in (0, T)$.   We define $\varphi(t)$ by $\omega(t)= \omega_t + \ddbar \varphi(t)$ and $\varphi(0)=0.$

Then $[\omega_t] = [\omega(t)]$. We define a perturbed K\"ahler metric by $\omega_{0, \epsilon} = \omega_0 + \epsilon \omega_+ $ with $\omega_+$ being a fixed K\"ahler metric on $X$ for $\epsilon\in (0, 1)$. We then consider the K\"ahler-Ricci flow on $X$ starting with $\omega_{0, \epsilon}$. Using the estimates as in \cite{SoT3}, one has the following proposition.

\begin{proposition} \label{locres}For each $\epsilon\in (0,1)$, there exists a unique smooth solution $g_\epsilon(t)$ of the K\"ahler-Ricci flow on $X$ starting from $\omega_{0, \epsilon}$ for $t\in [0, T)$. Fix $T' \in (0, T)$ and let $\varphi_\epsilon(t)$ be defined by
$$\omega_\epsilon(t) = \omega_{0, \epsilon} + t \chi + \ddbar \varphi_\epsilon(t), ~\varphi_\epsilon(0)= 0. $$ Then there exists $C>0$ such that $$||\varphi_\epsilon||_{L^\infty(X\times [0, T']) } \leq C$$ and for any compact subset $K \subset X\setminus E$
$$ \limsup_{\epsilon\rightarrow 0} || \varphi_{\epsilon} - \varphi ||_{C^k( K \times [0, T'] ) } =0, $$
where $C^k$ norm is taken with respect to the fixed K\"ahler metric $\omega_+$ on $X$.

\end{proposition}

An immediate consequence of Proposition \ref{locres} is the short time existence of the K\"ahler-Ricci flow on $X$ starting with the singular initial K\"ahler metric $\omega_0$ by the general results in  \cite{SoT3}. 
We let $X^+=X$ be a nonsingular model $\phi_+: X^+ \rightarrow Y$ of $Y$ so that $X^+$ is the flip of a minimal resolution $X^-$ of $Y$ over $p$. Then  the exceptional locus $E^+=(\phi_+)^{-1} (p)$ is isomorphic to $\mathbb{P}^l$ with $m=n-l-1$ and $m\geq  n/2$ whose normal bundle is $$ \mathcal{O}_{\mathbb{P}^l}(-1)^{\oplus(m+1)}.$$  Proposition \ref{locres} allows us to work locally near $E$ or simply a neighborhood of the zero section of the normal bundle  $\OO_{\PP^l}(-1)^{l+1}$.  We keep the same notations as in section 2. We also let $\tilde X$ be the blow-up of $X^+$ along the exceptional locus $E^+$ and let $\tilde \omega$ be a smooth K\"ahler metric on $\tilde X$.

Since $\omega_{0, \epsilon}$ is uniformly bounded above by a fixed Kahler metric $\omega_+$ on $X^+$, we can apply the same argument for Proposition \ref{keyest2} and obtain the following proposition.

\begin{proposition} \label{keyest3} Let $\omega_\epsilon(t)$ be the solution of the K\"ahler-Ricci flow on $X$ starting from $\omega_{0, \epsilon}$ for $\epsilon\in (0,1)$. There exist $\gamma>0$ and  $C>0$ such that for all $\epsilon\in [0,1)$, $t\in [0, T)$ ,  $p\in \Omega$,
\begin{equation} e^\rho \tr_{\hat\omega|_{L^-_{[\xi]}}}(\omega_\epsilon|_{L^-_{[\xi]}})  \leq C. \end{equation}
\begin{equation} e^\rho \tr_{\hat \omega|_{L^+_{[\eta]}}}(\omega_\epsilon |_{L^+_{[\eta]}}) \leq C. \end{equation}
\begin{equation} |W_-|^2_{\omega_\epsilon} + |W_+|^2_{\omega_\epsilon } \leq C e^{-\rho/2} . \end{equation}
 \begin{equation} e^{2\rho} \left( \tr_{\tilde \omega}(\omega_\epsilon ) \right)^\gamma  \tr_{\hat\omega  |_{L^-_{[\xi]}}}(\omega  |_{L^-_{[\xi]}}) \tr_{\hat \omega  |_{L^+_{[\eta]}}}(\omega_\epsilon |_{L^+_{[\eta]}}) \leq C, \end{equation}
where $\xi=\xi(p)$ and $\eta=\eta(p)$ as in (\ref{coorch1}) and (\ref{coorch2}).

\end{proposition}

By the same argument for Corollary \ref{estE}, we have the following corollary of Proposition \ref{keyest3}

\begin{corollary}\label{estE3}
Fix $T'\in (0, T)$ and let $\omega(t)$ be the solution of the K\"ahler-Ricci flow on $Y$ starting with $\omega_0$. There exists $C>0$ such that for all $t\in [0, T')$ and $p\in E^+$,
\begin{equation}
\omega(t)|_{E^+} \leq C \theta_+|_{E^+}.
\end{equation}

\end{corollary}

\bigskip

\noindent{\bf Proof of Theorem \ref{minus2}.}  Obviously, the metric completion of $(X\setminus E^+, g_0)$ is $(Y, g_0)$. $g(t)$ converges to $g_0$ smoothly on $X\setminus E^+$ as $t\rightarrow 0$.  With the estimates in Proposition \ref{keyest3}, we can obtain analogous geometric estimates as in Proposition \ref{onE}, Proposition \ref{geomest} and Corollary \ref{metrcom1} by the same argument. Theorem \ref{minus2} then follows immediately.

\qed

\bigskip

We will prove a useful and general volume estimate below. This will be useful to construct global flips in section 4.1.  In \cite{SoT3}, it is shown that on an $n$-dimensional  K\"ahler manifold $X$, the K\"ahler-Ricci flow has a unique solution starting from a K\"ahler current $\omega_0\in H^{1,1}(X^+, \mathbb{R}) \cap H^2(X^+, \mathbb{Z})$ with bounded potential and $L^p$ Monge-Amp\`ere current $\omega^n$ with respect to a fixed smooth volume form for some $p>1$.  Let $\rho'$ be a global smooth function on $X^+\setminus E^+$ with $\rho' = \rho$ on $\Omega$.

\begin{proposition}  \label{volatsin}

Let $\omega_+$ be a K\"ahler metric on $X^+$Suppose that $\omega_0$ is the pullback of a K\"ahler current on $Y$ with bounded potential. If
$$ \sup_{X^+} \frac{e^{(m-l) \rho' } \omega_0^n } {(\omega_+)^n}  < \infty, $$
Then the unique solution $g(t)$ of the K\"ahler-Ricci flow on $X$ with initial K\"ahler current $\omega_0$ on $(0, T)$ satisfies
$$\sup_{ (0, T')\times X^+} \frac{e^{(m-l)\rho'} \omega^n }{(\omega_+)^n} < \infty$$
for any $T'\in (0, T)$.

\end{proposition}

\begin{proof}  $X^+$ is a flip model of a smooth projective variety $X^-$ over $E^-\simeq \PP^m$ given by $\check\phi: X^- \dashrightarrow X^+$. Let $\omega_-$ be a smooth K\"ahler metric on $X^-$. Let $F=\frac{\omega_0^n}{ ( \omega_+)^n}$, then straightforward local calculations show that
$$\int_{X^+}  F^{1+\delta} (\omega_+)^n \leq C \int_{X^+} e^{-(1+\delta) (m-l)\rho'} (\omega_+)^n \leq C \int_{X^-} e^{-\delta (m-l) \rho'} ( \omega_-)^n< \infty$$
 for sufficiently small $\delta>0$.
One can approximate $F$ by a sequence of positive smooth functions $\{F_j \}_{j=1}^\infty$ in $L^{1+\delta}(X)$.

Let $\omega'_0$ be a smooth nonnegative closed $(1,1)$-form in the class of $[\omega_0]$.  Let $\omega_{0,j, \epsilon}$ be defined by the smooth K\"ahler form
$$\omega_{0, j, \epsilon}=\omega'_0 + \epsilon \omega_+ + \ddbar \varphi_{0, j, \epsilon}, ~\sup\varphi_{0,j,\epsilon}=0, ~~ ( \omega_{0, j, \epsilon} )^n = c_{j, \epsilon}F_j (\omega_+)^n ,$$ where $\epsilon\in (0,1)$ and $c_{j, \epsilon}$ is the normalization constant satisfying $c_{j, \epsilon} \int_{X^+} F_j (\omega_+)^n = [\omega_0 + \epsilon \omega_+]^n.$
Since $\omega_0$ is big and non-negative, $c_{j, \epsilon}$ is bounded from above and away from $0$ by constants independent of $j$ and $\epsilon$. From the uniform bound on the right side, there exists $C>0$ such that for all $j$ and $\epsilon \in (0, 1)$,
$$||\varphi_{0, j, \epsilon} ||_{L^\infty(X)} \leq C. $$
We then consider the smooth solution $\omega_{j,\epsilon}(t)$ of the K\"ahler-Ricci flow on $X$ starting with $\omega_{0, j, \epsilon}$. By \cite{SoT3}, $\omega_{j,\epsilon}(t)$ converges to the unique solution $\omega(t)$ of the K\"ahler-Ricci flow on $X$ starting with $\omega_0$, for $t\in [0, T)$, as $\epsilon \rightarrow 0$ and $j\rightarrow \infty$. Let $\Theta$ be a smooth volume form on $X^+$ and $\chi=\ddbar \log \Theta$. Then $\omega_{t, j, \epsilon}= \omega_{0, j, \epsilon} + t\chi \in [\omega_{j, \epsilon}(t)]$ is a smooth family of smooth closed $(1,1)$-forms on $X^+$. We let $\varphi_{j, \epsilon}(t)$ be the solution of $$\ddt{}\varphi_{j, \epsilon} = \log \frac{(\omega_{t, j,\epsilon}+ \ddbar \varphi_{j, \epsilon})^n}{\Theta}.$$ Then $\varphi_{j, \epsilon}\in C^\infty(X^+)\cap PSH(X^+, \omega_{t, j, \epsilon})$ for $t\in [0, T)$. Furthermore, by the estimates in \cite{SoT3}, $\varphi_{j, \epsilon}$ is bounded in $L^\infty([0,T']\times X^+)$ uniformly for $\epsilon\in (0,1)$ and all $j$.

  We let $H_{j,\epsilon }= \log \frac{e^{(m-l)\rho'} \omega_{j, \epsilon}^n}{\omega_X^n} -A\varphi_{j, \epsilon} $ Then straightforward calculations show that for sufficiently large $A>0$
$$ \Box_{j, \epsilon} H_{j, \epsilon} \leq  3nA$$
on $X^+$ for all $\epsilon\in (0, 1)$ and $t\in [0, T')$, where $\Box_{j, \epsilon}$ is the linearized operator of the K\"ahler-Ricci flow with respect to the evolving metrics $\omega_{j, \epsilon}$.
The maximum principle immediately implies that the upper bound for $\frac{ e^{(l-m)\rho'}\omega_{j, \epsilon}^n}{\omega_X^n}$ is independent on $j$ and $\epsilon$. The proposition is then proved by letting $j\rightarrow \infty$ and $\epsilon \rightarrow 0$.

\end{proof}

\section{Examples and generalizations}

In this section, we will first strengthen the result of \cite{SY} by constructing global flips  using the K\"ahler-Ricci flow on a family of projective manifolds admitting K\"ahler metrics with Calabi symmetry. Then we will prove Theorem \ref{main1} and Theorem \ref{main2} as the generalizations of Theorem \ref{minus1} and Theorem \ref{minus2} by using the constructions of Mumford's quotients.

\subsection{Flips by Ricci flow}

The surgical contractions with Calabi symmetry are  extensively studied in \cite{FIK, SW1, SY}. We first describe a family of projective bundles over $\mathbb{P}^m$. We let
$$X_{m, l} = \mathbb{P} ( \mathcal{O}_{\mathbb{P}^m}  \oplus \mathcal{O}_{\mathbb{P}^m} (-1)^{\oplus (l+1)})$$
be the projectivization of the bundle $\EE_{m,l}=\mathcal{O}_{\mathbb{P}^m}  \oplus \mathcal{O}_{\mathbb{P}^m} (-1)^{\oplus (l+1)}$ with $\mathbb{P}^{l+1}$ bundle being the fibre. We write $n=m+l+1$.  In particular, $X_{m, 0}$ is $\mathbb{P}^{m+1}$ blown up at one point. Let $D_\infty$ be the divisor in $X_{m, l}$ given by $\mathbb{P}(\mathcal{O}_{\mathbb{P}^m}(-1)^{\oplus (l+1)})$, the quotient of $\mathcal{O}_{\mathbb{P}^m}(-1)^{\oplus (l+1)}$. We also let $D_0$ be the divisor in $X_{m, l}$ given by  $\mathbb{P}(\mathcal{O}_{\mathbb{P}^m} \oplus \mathcal{O}_{\mathbb{P}^m} (-1) ^{\oplus l} )$, the quotient of $\mathcal{O}_{\mathbb{P}^m} \oplus \mathcal{O}_{\mathbb{P}^m} (-1) ^{\oplus l} $. In fact, $N^1(X_{m, l})$ is spanned by
$[D_0]$ and $[D_\infty]$. We also define the divisor $D_H$ on $X_{m, l}$ by the pullback of the divisor on $\mathbb{P}^m$ associated  to $\mathcal{O}_{\mathbb{P}^m}(1)$.
Then $$[D_\infty]= [D_0] + [D_H]$$ and
\begin{equation}\label{canbundle}
[ K_{X_{m, l}} ] = - (m+2) [D_\infty] - (m-l) [D_H] = -(m+2) [D_\infty] + (m-l) [D_0].
\end{equation}
In particular, $D_\infty$ is a big and semi-ample divisor and any divisor $ a[D_H] + b[D_\infty]$ is ample if and only if $a>0$ and $b>0$. Hence $X_{m, l}$ is Fano if and only if $m>l$.

Let  $E_{m, l}$ be  the zero section of $\pi_{m, l}: X_{m, l} \rightarrow \mathbb{P}^m$, which is the intersection of the $l+1$ effective divisors as the quotient of $\mathcal{O}_{\mathbb{P}^m} \oplus \mathcal{O}_{\mathbb{P}^m} (-1) ^{\oplus l}$ with its normal bundle $\mathcal{O}_{\mathbb{P}^m}(-1)^{\oplus (l+1)}. $  In fact, the linear system $| D_\infty|$ is base-point-free and it induces a morphism
\begin{equation}\label{contrmn}
\Phi_{m, l}: X_{m, l} \rightarrow \mathbb{P}^{(m+1)(l+1)-1}.
\end{equation}
$\Phi_{m, l}$ is an immersion on $X_{m, l}\setminus E_{m, l}$ and it contracts $E_{m, l}$ to a point. $Y_{m, l}$, the image of $\Phi_{m, l}$ in $\mathbb{P}^{(m+1)(l+1)-1}$, is a projective cone over $\mathbb{P}^m\times \mathbb{P}^l $ in $\mathbb{P}^{(m+1)(l+1)-1}$ by the Segre embedding
$$[Z_0, ..., Z_m]\times[W_0, ..., W_l]\rightarrow [Z_0W_0, ..., Z_iW_j, ..., Z_mW_l]\in \mathbb{P}^{(m+1)(l+1)-1} .$$
 Note that $Y_{m, l}=Y_{l,m}$.

The following diagram gives a flip from $X_{m, l}$ to $X_{l, m}$ for $0<l<m$,
\begin{equation}
\begin{diagram}\label{diagflip}
\node{X_{m, l}} \arrow{se,b,}{\Phi_{m, l}}  \arrow[2]{e,t,..}{ \check\Phi }     \node[2]{X_{l,m}} \arrow{sw,r}{\Phi_{l,m}} \\
\node[2]{Y_{m, l}}
\end{diagram}.
\end{equation}
Furthermore, $X_{m, l}$ and $X_{l,m}$ are birational to each other.

We now recall the Calabi ansatz constructed by Calabi \cite{C1} (also see \cite{Li, SY}) on $X_{m,l}$.  We instead consider the vector bundle $\EE_{m, l}= \mathcal{O}_{\mathbb{P}^m}(-1) ^{\oplus (l+1)}$ in order to apply the Calabi symmetry.

Let $\theta_{m, l}$ be the Fubini-Study metric on $\mathbb{P}^m$ in $\OO_{\PP^m}(1)$. Let $h$ be the hermitian metric on $\mathcal{O}_{\mathbb{P}^m} (-1)$ such that $\ric(h) = -\theta_{m,l}$. The induced hermitian metric $h_{\EE_{m,l}}$ on $\EE_{m,l}$ is given by $h_{\EE_{m,l}}= h^{\oplus (l+1)}$.
Under local trivialization of $E$, we write
$$e^\rho = h_{\EE_{m,l}} (z) |\xi|^2, ~~~ \xi = (\xi^0, \xi^1, ..., \xi^l),$$ where $h_{\EE_{m,l}} (z)$ is a local representation for $h$. Using the inhomogeneous coordinates $z=(z_1, z_2, ..., z_m)$ on $\mathbb{P}^m$, we have $h_{\EE_{m,l}} (z)= (1+|z|^2).$

We would like to find appropriate conditions for $a\in \mathbb{R}$ and a smooth real valued function $u=u(\rho)$ such that
\begin{equation}\label{metricrep1}
\omega = a \theta_{m,l} + \ddbar u(\rho)
\end{equation}
defines a K\"ahler metric on $X_{m, l}$. The following criterion is due to Calabi \cite{C1}.

\begin{proposition}\label{kacon}

$\omega$ as defined above is a K\"ahler metric if and only if

\begin{enumerate}

\item

$a>0$.

\item $u'>0$ and $u''>0$ for $\rho\in (-\infty, \infty)$.

\item $U_0 (e^\rho) = u(\rho)$ is smooth on $(-\infty, 0]$ and $U_0' (0)>0$.

\item $U_\infty (e^{-\rho}) = u(\rho) - b \rho$ is smooth on $[0, \infty)$ for some $b>0$ and $U_\infty'(0)>0$.

\end{enumerate}

\end{proposition}

By the calculations in \cite{SY}, if the initial K\"ahler class is $a_0[D_H]+ b_0[D_\infty]$ and evolving K\"ahler class along the K\"ahler-Ricci flow is $a(t)[D_H]+ b(t) [D_\infty]$, we have that
\begin{equation}
a=a(t) =a_0  - ( m-l) t,  ~ b=b(t) = b_0 - (m+2) t
\end{equation}
 and
\begin{equation}\label{krfu}
\ddt{u} = \log [(a+ u')^m (u')^l u''] - (l+1)\rho + c_t,
\end{equation}
where the normalization constant $c_t$ is given by
\begin{equation} \label{ct}
c_t = - \log u''(0,t) - l \log u'(0,t)- m\log (a(t)+u'(0,t)).
\end{equation}
The constant $c_t$ is chosen such that $\frac{\partial u}{\partial t}(0, t)=0.$

 It is straightforward to show that equation (\ref{krfu}) admits a smooth solution $u$ satisfying the Calabi ansatz as long as the K\"ahler-Ricci flow admits a smooth solution, by comparing $u$ to the solution of the Monge-Amp\`ere flow associated to the K\"ahler-Ricci flow. The evolution equations for $u'$ and $u''$ are given by
\begin{eqnarray}
 \ddt{u'} &=& \frac{u'''}{u''} + \frac{ l  u''}{u'} + \frac{ m u''}{a + u'} - (l+1),
 \\ \label{udpevolution}
\ddt{u''}& =& \frac{u^{(4)}} {u''} - \frac{ (u''')^2 }{ (u'')^2} + \frac{ l u'''}{u'} - \frac{ l (u'')^2}{ (u')^2} + \frac{ m u'''}{ a + u'} - \frac{ m (u'')^2}{( a + u')^2},
\end{eqnarray}
as can be seen from differentiating (\ref{krfu}) (see \cite{SW1, SY}).

The following theorem is proved in \cite{SY}, showing that the K\"ahler-Ricci flow with Calabi symmetry contracts $E_{m,l}$ if $0<l<m$ and the total volume does not go to zero as $t$ tends to the singular time.

\begin{theorem}\label{cycon} Let $g(t)$ be the smooth solution of the K\"ahler-Ricci flow on $X$ starting with a K\"ahler metric with Calabi symmetry. Suppose $0<l<m$ and $\limsup_{t\rightarrow T} \int_{X_{m,l}} dVol_{g(t)} >0$ as the flow develops singularity at $t=T>0$. Then $g(t)$ converges smoothly to a K\"ahler metric $g(T)$ on $X_{m,l} \setminus E_{m,l}$ and the metric completion of $(X_{m,l} \setminus E_{m,l}, g(T))$ is homeomorphic to $Y_{m,l}$. Furthermore, if we denote $(Y, d_T)$  the metric completion of $(X_{m,l}\setminus E_{m,l}, g(T))$, then
$(X_{m, l}, g(t))$ converges to $(Y_{m,l}, d_T)$ in Gromov-Hausdorff topology as $t \rightarrow T$.

\end{theorem}

Now we will show that the K\"ahler-Ricci flow smoothes out $(Y, d_T)$ instantly. The key observation is that the limiting K\"ahler current $g(T)$ is smooth on the regular part of $Y_{m,l}$ and it satisfies the Calabi symmetry on $X_{l, m}\setminus E_{l, m}$.

The limiting K\"ahler metric $\omega(T)$ is smooth on $X_{m, l}\setminus E_{m,l}$ with bounded local potential. Furthermore, its Monge-Amp\`ere current $(\omega(T))^n$  is bounded above by a smooth volume form on $X_{m, l}$ by the general theory in \cite{SoT3}  and so there exists  $C>0$ such that $(\omega(T))^n \leq C \max(1, e^{-(m-l)\rho}) (\omega_{l, m})^n$ for a fixed smooth K\"ahler form $\omega_{l,m}$ on $X_{l.m}$. This implies that $(\omega(T))^n/ (\omega_{l,m})^n$ is $L^p(X, (\omega_{l,m})^n)$ for some $p>1$. Therefore we can apply the results in \cite{SoT3} and continue the flow on $X_{m, l}$ weakly starting with $\omega(T)$. In particular, the solution becomes smooth instantly as $t>T$.

We let $g(t)$ be the unique smooth solution of the K\"ahler-Ricci flow on $X_{l, m}$ for $t\in (T, T']$. Since $\omega(T)$ satisfies the Calabi symmetry on $X_{l, m}\setminus E_{l, m}$, $\omega(t)$ can be obtained by a perturbation with Calabi symmetry and hence it satisfies the Calabi symmetry on $X_{l, m}$ for $t\in (T, T']$. We write $\omega(t)= \kappa (t) \theta_+ +\ddbar u(t, \rho)$, where $\kappa(t) = (m-l)(t-T)$ and $\theta_{l,m}$ is the Fubini-Study metric of $\PP^l$ in $\OO_{\PP^l}(1)$.

\begin{lemma} \label{u''0} There exists $C>0$ such that for all  $t\in  (T, T']$ and $\rho\in (-\infty, \infty)$,
\begin{equation}
 u' (t, \rho) \leq C \max\left( 1,e^{\frac{l+1}{m+l+1} \rho} \right).
\end{equation}

\end{lemma}

\begin{proof} It suffices to prove the estimate for $\rho\leq 1$ since the solution $g(t)$ of the Kahler-Ricci flow converges smoothly to $g(T)$ on $X_{l,m}\setminus E_{l,m}$ as $t\rightarrow T^+$. First we note that  there exists $C>0$ such that $$( \omega(T) ) ^n \leq C ( \omega_{m,l} )^n \leq C \max\left(1, e^{-(m-l) \rho} \right)\omega_{l, m}. $$
Then for $t\in (T, T']$, by Proposition \ref{volatsin},
$$ (\omega(t))^n \leq C e^{-(m-l) \rho} \omega_{l, m}$$
for $t\in (T, T']$ or equivalently, for $\rho\leq 0$,
$$( \kappa (t) + u ')^l (u ')^m u'' \leq C e^{-(m-l)\rho} (1+ e^\rho)^l e^{(m+1)\rho} . $$
This implies that for $\rho\leq 1$,
$$\left( (u ')^{m+l+1} \right)' \leq C e^{ (l+1) \rho}$$
and so
$$u' \leq C e^{ \frac{ (l+1)\rho}{ m+l+1} }$$
since $u(-\infty)=0.$

\end{proof}

The following lemma is proved in \cite{SY} (Proposition 3.1).

\begin{lemma}\label{u''}

There exists $C>0$ such that for all  $t\in  [0 , T)$ and $\rho\in (-\infty, \infty)$,
\begin{equation}
 u''(t, \rho) \leq Cu'(t, \rho).
\end{equation}

\end{lemma}

The following proposition shows that Lemma \ref{u''} also holds after the singular time.

\begin{proposition} \label{u''2}

There exists $C>0$ such that for all  $t\in  (T, T']$ and $\rho\in (-\infty, \infty)$,
\begin{equation}
 u''(t, \rho) \leq Cu'(t, \rho).
\end{equation}

\end{proposition}

\begin{proof} Let $u_{T, \epsilon} = u(T-\epsilon)$ for sufficiently small $\epsilon\in (0, T)$. Let $\omega_{T, \epsilon} = \epsilon \theta_{l,m} + \ddbar u_{T, \epsilon}$.  Then $\omega_{T, \epsilon}$ is a smooth K\"ahler metric on $X_{l, m}$ as $b(T)$ is strictly positive.  We then consider the K\"ahler-Ricci flow on $X_{l, m}$ starting with $\omega_{T, \epsilon}$. Then for $\epsilon$ sufficiently small, there is a unique smooth solution $\omega_\epsilon(t)$ of the K\"ahler-Ricci flow on $X_{l, m}$ for $t\in [T, T']$ starting with $\omega_{T, \epsilon}$. Since $\omega_{T, \epsilon}$ satisfies the Calabi symmetry, $\omega_\epsilon(t)$ also satisfies the Calabi symmetry and for $t\in [T, T']$,
$$ \omega_\epsilon = \kappa_\epsilon(t) \theta_{l,m} + \ddbar u_{\epsilon}, ~~\kappa_\epsilon(t)=\epsilon +(m-l)(t-T).$$
In addition, $u'_\epsilon$ and $u''_\epsilon$ satisfy the following parabolic equations
\begin{eqnarray}
 \ddt{u'_\epsilon} &=& \frac{u_\epsilon'''}{u_\epsilon''} + \frac{ m u_\epsilon''}{u_\epsilon'} + \frac{ l u_\epsilon''}{ \kappa_\epsilon + u_\epsilon'} - (m+1),
 \\ \label{udpevolution}
\ddt{u''_\epsilon}& =& \frac{u_\epsilon^{(4)}} {u_\epsilon''} - \frac{ (u_\epsilon''')^2 }{ (u_\epsilon'')^2} + \frac{ m u_\epsilon'''}{u_\epsilon'} - \frac{ m (u_\epsilon'')^2}{ (u_\epsilon')^2} + \frac{ l u_\epsilon'''}{ \kappa_\epsilon + u_\epsilon'} - \frac{ l (u_\epsilon'')^2}{( \kappa_\epsilon + u_\epsilon')^2},
\end{eqnarray}

We then follow the same argument as in \cite{SY, SW1}.  Let $H_\epsilon= \log u_\epsilon'' - \log u_\epsilon'$. Notice that by Proposition \ref{kacon}, for fixed $t\in[T, T']$
$$\lim_{\rho\rightarrow -\infty} \frac{u_\epsilon''}{u_\epsilon'}   =1,~  \lim_{\rho\rightarrow \infty}  \frac{u_\epsilon''}{u_\epsilon'}  = 0.$$ And so we can apply maximum principle for $H_\epsilon$ in $[T, T']\times (-\infty, \infty)$.

\begin{eqnarray*}
\ddt{H_\epsilon}
&=&  \frac{1}{u_\epsilon''} \left\{ \frac{ u_\epsilon^{(4)} }{u_\epsilon''} - \frac{ (u_\epsilon''')^2}{ (u_\epsilon'')^2} + \frac{   l u_\epsilon'''}{ u_\epsilon'} - \frac{ l (u_\epsilon'')^2}{ (u_\epsilon')^2} + \frac{ m u_\epsilon'''}{\kappa_\epsilon+ u_\epsilon'} - \frac{ m (u_\epsilon'')^2}{ (\kappa_\epsilon + u_\epsilon')^2} \right\} \\
&& - \frac{1}{u_\epsilon'} \{ \frac{u_\epsilon'''}{u_\epsilon''} + \frac{ l u_\epsilon''}{u_\epsilon'} + \frac{ m u_\epsilon''}{ \kappa_\epsilon + u_\epsilon'} - (l+1)\}
\end{eqnarray*}

Suppose that $H_\epsilon(t_0, \rho_0) = \sup_{[T, t_0] \times (-\infty, \infty) } H_\epsilon(t, \rho)$ is achieved for some $t_0\in [T, T']$, $\rho_0\in (-\infty, \infty)$.  At $(t_0, \rho_0)$, we have
$$H'_\epsilon= \frac{ u_\epsilon'''}{u_\epsilon''} - \frac{ u_\epsilon''}{u_\epsilon'} = 0, ~ H_\epsilon''= \frac{ u_\epsilon^{(4)}} {u_\epsilon''} - \frac{ (u_\epsilon''')^2}{(u_\epsilon'')^2} - \frac{ u_\epsilon''' } { u_\epsilon'} + \frac{ (u_\epsilon'')^2}{ (u_\epsilon')^2} = \frac{ u_\epsilon^{(4)} } {u_\epsilon''} - \frac{ u_\epsilon''' } { u_\epsilon'}\leq 0. $$
Then at $(t_0,\rho_0)$,
\begin{eqnarray*}
0 &\leq&\ddt{H_\epsilon}\\
&= & \frac{1}{u_\epsilon''} \{ \frac{ u_\epsilon^{(4)} }{u''} - \frac{ (u_\epsilon''')^2}{ (u_\epsilon'')^2} + \frac{ l u_\epsilon'''}{ u_\epsilon'} - \frac{ l(u_\epsilon'')^2}{ (u_\epsilon')^2}  \}  - \frac{1}{u_\epsilon'} \{ \frac{u_\epsilon'''}{u_\epsilon''} + \frac{ l u_\epsilon''}{u_\epsilon'}  - (l +1)\}  \\
&& + \frac{m}{\kappa_\epsilon +u_\epsilon'} \left\{  \frac{u_\epsilon'''}{u_\epsilon''} - \frac{u''}{\kappa_\epsilon+u_\epsilon'} - \frac{u_\epsilon''}{u_\epsilon'} \right\}  \\
&\leq&  - \frac{(l+1) u_\epsilon''}{(u_\epsilon')^2} + \frac{l+1}{u_\epsilon'}  - \frac{m u_\epsilon''}{(\kappa_\epsilon +u_\epsilon')^2}\\
&\leq&  \frac{l+1}{u_\epsilon'} (1- e^{H_\epsilon} ).
\end{eqnarray*}
Therefore by the maximum principle, $H_\epsilon(t_0, x_0) \leq 0$ and so there exists $C>0$ independent of $\epsilon\in(0,1)$ such that
$$\sup_{[T, T'] \times (-\infty, \infty)} H_\epsilon(t, \rho) \leq \sup_{(-\infty, \infty)} H_\epsilon(T, \rho) < C$$
since $H_\epsilon(T,\cdot)$ is uniformly bounded by Lemma \ref{u''}.
The proposition is then proved by letting $\epsilon \rightarrow 0$.

\end{proof}

We then have the following estimate for the evolving metric of the K\"ahler-Ricci flow on $X_{l,m}$ after the singular time $T$, by combining Lemma \ref{u''0} and Proposition \ref{u''2}.

\begin{corollary} There exists $C>0$ such that
\begin{equation} \label{resome}
\omega \leq C (\theta_{l,m} + \hat\omega +  e^{-\frac{m}{m+l+1}\rho} \hat\omega ),
\end{equation}
where $\hat\omega =\ddbar e^\rho$.

\end{corollary}

\noindent{\bf Proof of Theorem \ref{main3}.} We also note that for the smooth solution $\omega(t)=\kappa(t)\theta_{l,m} +\ddbar u(t, \rho)$ for $t\in (T, T']$, $\kappa(t)$ tends to $0$ as $t\rightarrow T^+$. With the estimates (\ref{resome}), one can apply the arguments as in \cite{SY} or the proof of Theorem \ref{minus2} to show that $(X_{l,m}, g(t))$ converges to $(Y, d_T)$ in Gromov-Hausdorff topology as $t\rightarrow T^+$. In particular, the convergence is smooth on $X_{l,m}\setminus E_{l,m}$. This, combined with Theorem \ref{cycon}, proves Theorem \ref{main3} by verifying the definition for a surgical metric flip.

\qed

\subsection{Complex cobordisms }

We first recall the definition of birational cobordisms introduced by Wlodarczyk \cite{W}.
\begin{definition} Let $\phi: X_1: \dasharrow X_2$ be a birational map between two $n$-dimensional nonsingular projective varieties $X_1$ and $X_2$. The birational cobordism   is an $(n+1)$-dimensional variety $B=B_\phi(X_1, X_2)$ with an algebraic action of the $\CC^*$-action satisfying
\begin{enumerate}
\item  the sets
\begin{eqnarray*}
&& B^-=\{z\in B ~|~ \lim_{\lambda \rightarrow 0} \lambda \cdot z~\textnormal{ does~not~exist}  \}\\
&& B^+=\{z\in B~|~\lim_{\lambda\rightarrow \infty} \lambda\cdot z ~\textnormal{does~ not~exist} \}.
\end{eqnarray*}
are nonempty and open;
\item there exist quotients of $B^-$ and $B^+$ by the $\CC^*$-action such that $X_1\simeq B^-/\CC^*$ and  $X_2\simeq B^+/ \CC^*$.

\end{enumerate}

\end{definition}

The complex cobordisms can be viewed as an analogue of the cobordism between differential manifolds  by a Morse function. The 'passing-through' fixed points  of the $\CC^*$-action  on $B^-$ and $B^+$ induces a birational transformation. Such a point of view leads to the fundamental factorization theorem \cite{W}, stating that any birational map between nonsingular varieties is a composition of  blow-ups and blow-downs along nonsingular centers.

 Mumford's quotients can be used  to understand a family of flips. We consider the $\CC^*$ action on $\BB=\CC^{m+l+2}$ defined by
$$ \CC^*: (t, (x_0, ..., x_m; y_0, ..., y_l)) \rightarrow ( \lambda^{-a_0} x_0, \lambda^{-a_1} x_1, ..., \lambda^{-a_m} x_m; \lambda^{b_0} y_0, ..., \lambda^{b_l} y_l ), $$
where $a_0$, ... , $a_m$, $b_0$, ... , $b_l \in \mathbb{Z}^{+}$.
We define
$$\BB^-=\left( \CC^{m+l+2}\setminus \{ x=0\} \right) / \CC^*, ~\BB^+ =\left( \CC^{m+l+2}\setminus \{y=0\} \right) / \CC^*. $$
Let $\EE^-= \BB^- /\CC^*$ and $\EE^+=\BB^+/ \CC^*$. Then the triple $(\BB, \EE^-, \EE^+ )$ induces a birational coboridsm. The fixed points of the $\CC^*$-action on $\BB^-$ and $\BB^+$ are weighted projective spaces $E^-= \PP^m_{(a_0, ..., a_m)}$ and $E^+= \PP^l_{(b_0, .., b_l)}$. In general, $\EE^-$ and $\EE^+$ have orbifold singularities. From now on, we assume that
$$g.c.d(a_0, ..., a_{i-1}, a_{i+1}, ..., a_m; b_0, ..., b_l)= g.c.d. (a_0, ..., a_m; b_0, ..., b_{j-1}, b_{j+1}, ... b_l)=1$$
for all $i=0, ..., m$ and $j=0, ..., l$.

\bigskip

\noindent{\bf Examples} Let us illustrate how the Mumford's quotients are related to divisorial contractions and flips by the following examples.

\begin{enumerate}

\item When $a_0=a_1=...=a_m=b_0=...=b_l=1$,
if $l=0$,
$$\EE^- = \OO_{\PP^m}(-1), ~\EE^+ =\CC^{m+1}. $$
and if $l\geq 1$,
$$\EE^- = \OO_{\PP^m}(-1)^{\oplus (l+1)}, ~\EE^+ = \OO_{\PP^l}(-1)^{\oplus (m+1)}. $$ The K\"ahler quotient gives the local model for a flip if $0<l<m$ and for a flop if $m=l>0$.

\medskip

\item When $a_0=-2, a_1=-1$ and $b_0=b_1=1$.
$\EE^-$ has two coordinate patches $\{x_0 \neq 0\} = \CC^3/ \mathbb{Z}_2=$ and $\{x_1\neq 0\} = \CC^3$.
Let $(z_1, z_2, z_3)$ be the orbifold coordinates on $\{x_0 \neq 0\} $ with group action $(-1,-1,-1)$ on $\CC^3$ and $(w_1, w_2, w_3)$ be the coordinates on $\{x_1\neq 0\} = \CC^3$. Then transition functions are given by
$$ w_1= 1/z_1, ~ w_2 = z_2/z_1, ~ w_3 = z_3/z_1. $$
In addition, $\EE^+ = \OO_{\PP^1} (-2) \oplus \OO_{\PP^1}(-1)$. This gives an orbifold  flip in dimension 3.

\medskip

\item When $a_0=a_1=...=a_m=1$ and $\sum b_j \leq m$,  $\EE^-= \oplus \OO_{\PP^m}(-b_j)$ and $\EE^+$ is the flip of $\EE^-$ with orbifold singularities.

\end{enumerate}

\bigskip

For each coordinate patch $x_i\neq 0$, the quotient $\EE^-$  can be viewed as a subset of a quotient of $\OO_{\PP^m}(-1)^{\oplus (j+1)}$.
We consider the following function analogous to the Morse function
\begin{equation} e^\Upsilon = ( \sum |x_i|^{2/a_i}) (\sum |y_j|^{2/b_j}). \end{equation}
Similarly as in section 2, we define
\begin{equation}
\Omega = \{ (x;y)\in \BB^-\cap \BB^+~|~e^\Upsilon \leq 1 \} / \CC^*.
\end{equation}
Then we define
\begin{equation}
\hat\omega = \ddbar e^\Upsilon, ~\theta_- = \ddbar \log (  \sum |x_i|^{2/a_i}), ~ \theta_+ = \ddbar \log  ( \sum |y_j|^{2/b_j}). \end{equation}
$\theta_-$ and $\theta_+$ are smooth orbifold K\"ahler metrics on $\PP^m_{(a_0, ..., a_m)}$ and $\PP^l_{(b_0, ..., b_l)}$,  which can be identified as the  quotient of the Fubini-Study metrics on $\PP^m$ and $\PP^l$.

We now compare smooth metrics on $\EE^-$ and $\EE^+$ to $\hat\omega$, $\theta_-$ and $\theta_+$.

\begin{lemma} Suppose $\omega_-$ and $\omega_+$ are smooth orbifold K\"ahler metrics  on $\EE^-\cap \Omega$ and $\EE^+\cap \Omega$ respectively, then there exists $C>0$ such that
\begin{equation} \label{coboest1}
\omega_- \leq C ( \theta_-+ \hat\omega ), ~~ \omega_+ \leq C ( \theta_+ + \hat\omega ),
\end{equation}
\begin{equation}\label{coboest2}
 \omega_- \leq C e^{-\Upsilon } \hat\omega, ~ \omega_+ \leq C e^{-\Upsilon } \hat\omega.
\end{equation}

\end{lemma}

\begin{proof}

We consider the coordinate patch
$$U^-_i = \left\{ (x;y)~|~|x_i|^{2/a_i} > \frac{1}{2(m+1)} \sum_{l=0}^m |x_l|^{2/a_l}  \right\} / \CC^*. $$ It is isomorphic to the quotient of $\CC^{m+l+1}$ by $\mathbb{Z}_{a_i}$ acting on $\CC^{m+l+1}$ by 
$$\theta: (z_0, ..., z_{i-1}, z_{i+1}, ..., z_m, w_0, ..., w_l) \rightarrow (\theta^{a_0} z_0, ..., \theta^{a_{i-1}} z_{i-1}, \theta^{a_{i+1}} z_{i+1}, ..., \theta^{a_m}z_m, \theta^{b_0}w_0, ..., \theta^{b_l} w_l)
$$
for any $a_i $ unit root $\theta$.
We set $z_0$, ... , $z_{i-1}$, $z_{i+1}$,  $z_m$, $w_0$,  ... ,  $w_l$ to be the coordinates on $\CC^{m+l+1}$ satisfying
$$ |z_0|^2= \frac{|x_0|^2}{ |x_i|^{2a_0/a_i} }, ..., |z_m|^2= \frac{|x_m|^2}{ |x_i |^{2a_m/a_i} }, |w_0|^2= |y_0|^2|x_i|^{2b_0/a_i} , ..., |w_l|^2= |y_l |^2 |x_i|^{2b_l/a_i}  . $$
Immediately we have
$$\ddbar |w_j|^2  = \sum \ddbar \left (|x_i|^{2/a_i} |y_j|^{2/b_j}\right)^{b_j} \leq C  \sum \ddbar  \left( |x_i|^{2/a_i} |y_j|^{2/b_j} \right)\leq C\hat\omega ,$$
$$\ddbar \sum_{l\neq i} |z_l|^2 = \ddbar \sum_{l\neq i} \left( \frac{|x_l|^{2/a_l}}{ |x_i|^{2/a_i} } \right)^{a_l}\leq C\ddbar \sum_{l\neq i} \left( \frac{|x_l|^{2/a_l}}{ |x_i|^{2/a_i} } \right) \leq C \theta_-. $$
This proves (\ref{coboest1}) as the same argument applies to $\omega_+$.

To prove (\ref{coboest2}), we use the fact that $\OO_{\PP^m}(-1)^{\oplus (l+1)}$ is an orbifold covering of $U^-_i$ by $\mathbb{Z}_{a_i}$.  Then we can apply the estimate in Lemma \ref{lcompa} because $\omega_-$ is bounded above by a fixed K\"ahler metric on $\OO_{\PP^m}(-1)^{\oplus (l+1)}$, and the lifting of $e^\Upsilon$, $\theta_-$ and $\hat\omega$ coincide with $e^\rho$, $\theta_-$ and $\hat\omega$ defined in section 2. Then (\ref{coboest2}) follows as $e^\Upsilon \theta_-$ is bounded above by a multiple of $\hat\omega$.

\end{proof}

We now consider the  K\"ahler-Ricci flow on a projective variety $X^-$ and let $g(t)$ be a smooth orbifold solution for $t\in [0, T)$. Assume the flow develops singularity at $t=T< \infty$ and $\lim_{t\rightarrow T} [g(t)]\in H^{1,1}(X, \mathbb{R}) \cap H^2(X, \mathbb{Z})$ induces a flip
$$\phi_-: X^- \rightarrow Y \leftarrow X^-: \phi_+$$
of $X^-$ such that  the flip over each exceptional locus is isomorphic to a Mumford quotient.  Without loss of generality, we can assume that there exists only one component for the exceptional locus $E^-$. Then by the general result of \cite{SoT3}, the K\"ahler-Ricci flow converges smoothly to a K\"ahler metric on $X\setminus E^-$. We can can localize the estimates in a neighborhood of $E^-$ and by choosing suitable coordinates. For any point $\alpha=[\alpha_0, ..., \alpha_m] \in \PP^m_{(a_0, ..., a_m)}$ and  $\beta=[\beta_0, ..., \beta_l] \in \PP^l_{(b_0, ..., b_l)}$,  we define
\begin{eqnarray*}
L^-_{[\beta]} &=& \{ (\lambda^{a_0} x_0, ..., \lambda^{a_m} x_m;  \beta_0, ...,  \beta_l )~|~\lambda\in \CC^*, ~x\in \CC^m \}/ \CC^*   \\
 &=& \{ ( x_0, ...,  x_m;  \lambda^{b_0}\beta_0, ..., \lambda^{b_l} \beta_l )~|~\lambda\in \CC^*, ~x\in \CC^m \}/ \CC^*, \\
L^+_{[\alpha]} &=& \{ ( \alpha_0, ...,  \alpha_m; \lambda^{b_0} y_0, ..., \lambda^{b_l} y_l ) ~|~\lambda\in \CC^* ,~ y \in \CC^{l+1}  \}/  \CC^*\\
&=& \{ ( \lambda^{a_0} \alpha_0, ...,  \lambda^{a_m} \alpha_m;  y_0, ...,  y_l ) ~|~\lambda\in \CC^* ,~ y \in \CC^{l+1}  \}/  \CC^*.
\end{eqnarray*}

At each point $(x; y)$ with $x\neq 0, y\neq 0$, there exists a unique branch of $L^-_{[x]}$ and $L^+_{[y]}$ passing through $(x; y)$. When $a_0=...=a_m=b_0=...=b_l=1$, we simply have $$L^-_{[y]} = \OO_{\PP^m} (-1), ~ L^+_{[x]} = \OO_{\PP^l} (-1). $$

On $L^-_{[\beta]}$, suppose $\beta_0 \neq 0$, then we can choose coordinates $(\nu_0, \nu_1, ..., \nu_m)$  on the covering space of $L^-{[\beta]}$, such that
\begin{equation}\label{coordch}
|\nu_i |^2 = |x_i|^{2/a_i} |y_0|^{2/b_0}.
\end{equation}
Direct calculations show that the restriction of $\hat\omega$ on $L^-_{[\beta]}$ and $L^+_{[\alpha]}  $ are induced by standard Euclidean metrics on $\CC^{m+1} $ and $\CC^{l+1}  $.
\begin{lemma} $\hat\omega|_{L^-_{[\beta]}} $ and $\hat\omega|_{L^+_{[\alpha]}} $ are  flat for all $[\beta]\in \PP^l_{[b_0, ..., b_l]}$ and $[\alpha]\in \PP^m_{[a_0, ..., a_m]}$.

\end{lemma}

The holomorphic vector field
$$V_- = \sum b_j^{-1} y_j \frac{\partial }{\partial y_j} $$  on $\CC^{m+l+2}$ induces  a holomorphic vector field on $\EE^-\cap \Omega$.  $V_-$ is tangential to $L^-_{[\beta]}$ for all $[\beta]\in \PP^{l}_{[b_0, ..., b_l]}$.
For fixed $[\beta]\in \PP^{l}_{[b_0, ..., b_l]}$ with  $\beta_0\neq 0$, we choose  coordinates $(\nu_0, \nu_1, ..., \nu_m)$ on $L^-_{[\beta]}$ as in (\ref{coordch}), and by direction calculations, we have
$$ V_- |_{L_{[\beta]}} = b_0^{-1} \sum_{i=0}^m \nu_i \frac{\partial }{\partial \nu_i}. $$
Similarly, on $\EE^+\cap \Omega$, we define the holomorphic vector field $V_+$ induced by the holomorphic vector field
$$\sum a_i^{-1} x_i \frac{\partial }{\partial x_i} $$
and  $V_+$ is tangential to $L^+_{[\alpha]}$ for all $[\alpha]\in \PP^{m}_{[a_0, ..., a_m]}$.
After lifting to a cover, we can apply the same calculations in section 2, we can show that  $|V_-|^2_{\hat\omega} = |V_+|^2_{\hat\omega} = e^\Upsilon$ . We then define
$$W_- = e^{-\Upsilon/2} V_-, ~ W_+ = e^{-\Upsilon/2} V_+. $$

\begin{proposition} \label{53} There exist $\gamma>0$ and  $C>0$ such that for all $t\in [0, T)$ ,  $p\in \Omega$  with $(x; y)=(x(p) ; y(p))\in \CC^{m+l+2}$,

\begin{equation} e^\Upsilon \tr_{\hat\omega|_{L_{[y]}}}(\omega|_{L_{[y]}})  (t, p) \leq C. \end{equation}

\begin{equation} e^\Upsilon \tr_{\hat \omega|_{L_{[x]}}}(\omega|_{L_{[x]}} ) (t, p) \leq C. \end{equation}

\begin{equation} |W_-|^2_{g(t)} (t, p) + |W_+|^2_{g(t)}  (t, p) \leq Ce^{-\Upsilon/2} . \end{equation}

\end{proposition}

\begin{proof}  The proof proceeds the same way the proof for Proposition \ref{keyest3}.

\end{proof}

Proposition \ref{53} immediately gives the estimates of evolving metrics in the normal directions of $E^-$ and implies a uniform diameter for $(X, g(t))$. The same argument for Corollary \ref{corestE} gives the following estimate for the evolving metric on $E^-$.
\begin{corollary}
There exists $C>0$ such that for all $t\in [0, T)$,
\begin{equation}\label{estE3}
\omega(t)|_{E^-} \leq C \theta_- |_{E^-}.
\end{equation}

\end{corollary}

We can now state the following  generalizations of Theorem \ref{minus1} and Theorem \ref{minus2}. In particular,  Theorem \ref{main1} and \ref{main2} follow immediately.

 \begin{theorem} \label{gen1} Let $X$ be a projective orbifold of $\dim_{\CC} X =n$ and let $g(t)$ be a smooth solution of the K\"ahler-Ricci flow for $t\in [0, T)$, starting from a smooth K\"ahler metric $g_0$ with $[g_0] \in H^{1,1}(X, \mathbb{R})\cap H^2(X, \mathbb{Q})$.  Suppose that

 \begin{enumerate}

\item the limiting K\"ahler class $\lim_{t\rightarrow T} [g(t)] $ is a Cartier $\mathbb{Q}$-divisor  which induces a birational mophism $\phi: X \rightarrow Y$ from $X$ to a normal projective variety $Y$;

\item locally $\phi: X\rightarrow Y$ is a divisorial contraction or a small contraction  isomorphic to a Mumford's quotient.

\end{enumerate}
Let $E$ be the exceptional locus of $\phi$. Then the following holds.
\begin{enumerate}

\item $g(t)$ converges to a smooth K\"ahler metric $g(T)$ on $X\setminus E$ in $C^\infty(X\setminus E)$.

\item The metric completion of $(X\setminus E, g(T))$ is a compact metric length space homeomorphic to $Y$. We denote it by $(Y, d_T)$

\item $(X, g(t))$ converges in Gromov-Hausdorff topology to $(Y, d_T)$ as $t\rightarrow T$.

\end{enumerate}

 \end{theorem}

\begin{proof} We follow the same argument in section 3.2. Proposition \ref{onE}, Proposition 3.6 and Corollary \ref{metrcom1} can be proved because all the computation and estimates can be achieved by locally lifting to a copy of $\OO_{\PP^m}(-1)^{\oplus(l+1)}$.

\end{proof}

\begin{theorem} \label{gen2}  Let $Y$ be a projective variety of $\dim_{\CC} Y=n$ with isolated singularities $p_1, ..., p_k$. Suppose

\begin{enumerate}


\item there exists a flip $\phi_-: X^- \rightarrow Y \leftarrow X^+: \phi_+$ such that $\phi_+$ is a small resolution of singularities along $p_1, ..., p_k$;

\item the flip $(X^-, X^+)$ is locally isomorphic to a Mumford's quotient.

\end{enumerate}
Let $E^+$ be the exceptional locus of $\phi_+$. Let $g_0$ be a smooth K\"ahler metric on $Y$, i.e., locally $g_0$ is the restriction of a smooth K\"ahler metric on a local embedding of $Y$ in $\CC^N$ for some $\CC^N$. Then there exists a unique smooth orbifold solution $g(t)$ of the K\"ahler-Ricci flow on $X$ for $t\in (0, T)$ for some $T\leq \infty$ satisfying

\begin{enumerate}

\item $g(t)$ converges to  $g_0$ on $X\setminus E^+$ in $C^\infty(X\setminus E^+)$.

\item $(X, g(t))$ converges in Gromov-Hausdorff topology to $(Y, g_0)$ as $t\rightarrow 0$.

\end{enumerate}

\end{theorem}

\begin{proof} The same argument for the proof of Theorem \ref{minus2} applies since near each $p_i$, all the computation and estimates can be achieved on a local lifting.

\end{proof}

The above arguments can be easily applied for small contractions and flops of Calabi-Yau orbifolds if locally the contractions and flops are isomorphic to Mumford's quotients. This generalizes the results in \cite{So3}.

\bigskip

\section{Blow-up models and a metric uniformization}

In this section, we will apply the results of \cite{FIK} and \cite{Li} to construct the local models of the flip.
We consider a local model for flips given by
$$\phi_-: \EE^- \rightarrow \hat \EE \leftarrow \EE^+: \phi_+$$
where $\EE^-= \OO_{\PP^m}(-1)^{\oplus (l+1)}$ and $\EE^+ = \OO_{\PP^l}(-1)^{\oplus (m+1)} $ are defined as in section 2 for $0<l<m$.

 We keep the same notations as in section 2. Note that $\hat\EE$ is a sub-cone in $\CC^{(m+1)(l+1)}$ induced by the Segre map from $\PP^m \times \PP^l$ to $\CC^{(m+1)(l+1)}$ as in (\ref{segre}).

Let $\zz=(\zz_1, ..., \zz_{(m+1)(l+1)})$ be the standard coordinates on $\CC^{(m+1)(l+1)}$ We consider the cone metric $g_{  \gamma}$ on $\CC^{(m+1)(l+1)}$ defined by
\begin{eqnarray*}
g_{\gamma } &=& \ddbar  (|\zz|^{2\gamma } )\\
&=& \sqrt{-1} |\zz|^{2\gamma -2} \left( \delta_{\alpha \bar \beta} + (\gamma-1) |\zz|^{-2}\bar\zz_\alpha \zz_\beta \right) d \zz_\alpha\wedge d\bar \zz_\beta.
\end{eqnarray*}
Obviously, the restriction of $g_{\gamma}$ on $\hat\EE$ is also a cone metric and it satisfies the Calabi symmetry with
$$ g_{\gamma}|_{\EE^-} = \ddbar ( e^{\gamma \rho})$$
since $|\zz|^2|_{\EE^-} = e^{\rho}$.

It is shown in \cite{Li} that there exists a unique complete shrinking gradient K\"ahler-Ricci soliton $g_{KR,-}$ on $\EE^-$ such that $g_{KR, -}$ satisfies the Calabi symmetry with $g_{KR}|_{E^-} =\theta_-$. In particular, the asymptotic cone of $g_{KR,-}$ at infinity is given by the cone metric $g_{cone}$ on $\EE^-$, where $g_{cone}=g_{\gamma }|_{\EE^-}$ for a unique  $\gamma=\gamma(m,l) \in (0,1)$. $\gamma$ can be calculated using the algebraic equation (18) in \cite{Li}.

Then by \cite{Li}, there also exists a unique complete expanding gradient K\"ahler-Ricci slotion $g_{KR, +}$ on $\EE^+$ such that $g_{KR, +}$ satisfies the Calabi symmetry with $g_{KR,+}|_{E^+} = \theta_+$ and the asymptotic cone of $g_{KR, +} $ at infinity coincides with $g_{cone}$.

$g_{KR, -}$ induces $g(t)$, a solution of the K\"ahler-Ricci flow on $\EE^-$ for $t<0$, and $g_{KR, +}$ induces $g(t)$, a solution of the K\"ahler-Ricci flow on $\EE^+$ for $t>0$, such that

\medskip

\begin{enumerate}

\item $g(t)$ converges to $g_{cone}$ smoothly on $\EE^-\setminus E^-$, as $t\rightarrow 0^-$;

\medskip

\item $g(t)$ converges to $g_{cone}$ smoothly on $\EE^+ \setminus E^+$, as $t\rightarrow 0^+$;

\medskip

\item  $(\EE^-, p_-,  g(t))$ converges to $(\hat\EE, g_{cone})$ in pointed Gromov-Hausdorff topology as $t\rightarrow 0^-$ for any fixed point $p_-\in E^-$;

\medskip

\item  $(\EE^+, p_+, g(t))$ converges to $(\hat\EE, g_{cone})$ in Gromov-Hausdorff topology as $t\rightarrow 0^+$ for any fixed point $p_+\in E^+$.

\end{enumerate}
This immediately implies Proposition \ref{main4}. We conjecture that such a local model for a surgical metric flip is exactly the blow-up limit for a global surgical metric flip by the Ricci flow. The global surgical metric flip by the K\"ahler-Ricci flow should be indeed a self-similar metric surgery after parabolic Type-I scaling. The small contraction before the singular time together with the small resolution after the singular time is modeled on a continuous path in pointed Gromov-Hausdorff topology, joining a complete shrinking gradient K\"ahler-Ricci soliton and a complete expanding gradient K\"ahler-Ricci soliton through the metric tangent cone at the singularity formed at the singular time.

We therefore make the following conjecture.

\begin{conjecture} \label{conj1} Let $\phi_-: X^- \rightarrow Y \leftarrow X^+ : \phi_+$ be a smooth flip of two smooth projective manifolds. Let $g(t)$ be a smooth solution of the K\"ahler-Ricci flow on $X^-$ for $t\in [0, T)$. Suppose $\lim_{t\rightarrow T} [g(t)]$ is the pullback of an ample class on $Y$, then the K\"ahler-Ricci flow performs a surgical metric flip at $t=T$. Furthermore, let $g(t)$ be the solution through singularities at $t=T$, then
$$ \limsup_{t\rightarrow T^-} (T-t) |Rm(g(t))|_{g(t)} < \infty, ~ \limsup_{t\rightarrow T^+} (t-T) |Rm(g(t))|_{g(t)} < \infty. $$
In particular, after suitable parabolic Type-I rescaling, the blow-up limits of $g(t)$ through the singular time $T$ is a continuous path in pointed Gromove-Hausdorff topology, joining a complete shrinking gradient K\"ahler-Ricci soliton and a complete expanding gradient K\"ahler-Ricci soliton through the metric tangent cone of the singularities formed at $t=T$.

\end{conjecture}

When the exceptional loci of $\phi_-$ and $\phi_+$ are nonsingular, the solitons that appear as the blow-up limit, should live on the normal bundle on the contracted fibre. Conjecture \ref{conj1} can also be generalized to flips for projective varieties with log terminal singularities as there exists a unique analytic solution of the weak K\"ahler-Ricci flow on such varieties \cite{SoT3}.  Of course, one does not expect the Type-I bound for the full curvature tensors to hold because the underlying variety is already singular, however, one might hope that after suitable pointed Type-I parabolic dilations, the flow would still converge to a complete shrinking and expanding gradient solitons on the cotangent sheaf of the fibre of $\phi_-$ and $\phi_+$ through the singular time.

In the following, we make a speculation on the metric uniformation for all algebraic singularities arising from smooth flips.

\begin{conjecture} [Metric uniformization for flips]  \label{conj2} Let $\phi_-: X^- \rightarrow Y \leftarrow X^+ : \phi_+$ be a smooth flip of two smooth projective manifolds.    Then at each point $y$ in the singular set of $Y$, $p^-\in (\phi_-) ^{-1}(y)$ and $p^+\in (\phi_+)^{-1}(y)$, there exist

\begin{enumerate}

\item  a unique complete shrinking K\"ahler-Ricci soliton $g^-(t)$ on $N_{-}$ the normal bundle of $(\phi_-)^{-1} (y) $ for $t\in (-\infty, 0)$,

\medskip

\item a unique complete expanding K\"ahler-Ricci soliton $g^+(t)$ on $N_+$ the normal bundle of $(\phi_+)^{-1}(y) $ for $ t\in (0, \infty)$,

\medskip

\item a cone metric $g_Y$ on the tangent cone of $Y$ at $y$,

\end{enumerate}
such that  $(N_-, p^-, g^-(t))$ converges to $(Y, y, g_Y)$ as $t\rightarrow 0^-$ and $(N_+, p^+, g^+(t))$ converges to $(Y, g_Y)$ as $t\rightarrow 0^+$, in pointed Gromov-Hausdorff topology. Furthermore, the converges is smooth outside $E^-$ and $E^+$, the exceptional locus of $\phi_-$ and $\phi_+$.

\item

\end{conjecture}
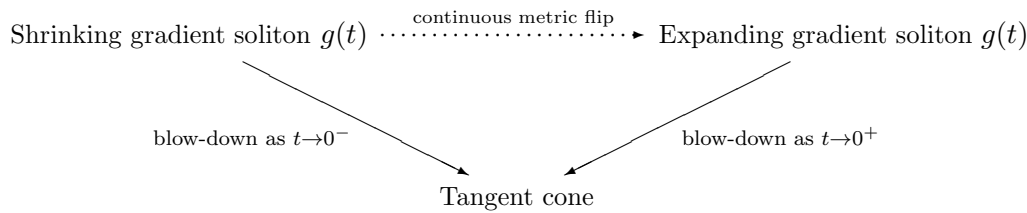
$$\begin{diagram}
\node{ \textnormal{\small Shrinking~gradient~soliton} ~g(t)} \arrow{se,b,}{\textnormal{blow-down~as~} t\rightarrow 0^-}  \arrow[2]{e,t,..}{ \textnormal{\tiny continuous~metric~flip} }     \node[2]{\textnormal{\small Expanding~gradient~soliton~} g(t)} \arrow{sw,r}{ \textnormal{blow-down~as~} t\rightarrow 0^+} \\
\node[2]{\textnormal{\small Tangent ~cone}}
\end{diagram}$$

Conjecture \ref{conj2} can also be generalized to flips of two projective varieties with log terminal singularities with suitable modifications. In principle, the K\"ahler-Ricci flow should lead to a metric uniformization for all projective varieties in the following sense. Let $X$ be a projective variety with log terminal singularities and positive Kodaira dimension. Then after finitely many metric surgerical divisorial contractions and flips by the K\"ahler-Ricci flow, $X$ is replaced by its minimal model and afterwards the K\"ahler-Ricci flow will eventually converge to its canonical model coupled with a unique canonical K\"ahler metric of Einstein type. In particular, each surgery is composed with a complete shrinking and expanding K\"ahler-Ricci soliton through the tangent cone of the algebraic singularities arising from the divisorial contraction or flip. The long time behavior of the K\"ahler-Ricci flow is studied and discussed in \cite{SoT1, SoT2, SoT3, SoT4, GTZ}.


\bigskip
\bigskip

\noindent {\bf{Acknowledgements:}} The author would like to thank D.H. Phong,  Gang Tian, Tom Ilmanen, Ben Weinkove, Yuan Yuan and Xiaowei Wang for many stimulating discussions. He would also like to thank Ved Datar and Bin Guo for a number of helpful suggestions.


\bigskip
\bigskip

\end{document}